\documentclass[a4paper]{amsart}

\usepackage[backend=bibtex,style=alphabetic,natbib=true,maxbibnames=99, minalphanames=1, maxalphanames=4]{biblatex} %

\usepackage{tikz}
\usetikzlibrary{decorations.pathreplacing}
\usepackage{amsmath}
\usepackage{amssymb}
\usepackage{mathrsfs}
\usepackage[shortlabels]{enumitem}
\usetikzlibrary{arrows,cd,positioning}

\usepackage[utf8]{inputenc}
\usepackage{xcolor}
\usepackage{geometry}
\usepackage{graphicx}
\usepackage{mathtools}
\usepackage{esint}
\usepackage{comment}
\usepackage[colorlinks=true,linkcolor=blue]{hyperref}

\newtheorem{theorem}{Theorem}[section]
\newtheorem{lemma}[theorem]{Lemma}
\newtheorem{corollary}[theorem]{Corollary}
\newtheorem{proposition}[theorem]{Proposition}

\theoremstyle{definition}
\newtheorem{definition}[theorem]{Definition}

\theoremstyle{remark}
\newtheorem{remark}[theorem]{Remark}

\numberwithin{equation}{section}

\newcommand {\R} {\mathbb{R}} \newcommand {\Z} {\mathbb{Z}}
 \newcommand {\N} {\mathbb{N}}
 
\newcommand {\p} {\partial}

\newcommand {\D} {\Delta}

\newcommand {\supp} {\textup{supp}}
\newcommand {\diam} {\textup{diam}}

\newcommand\restr[2]{{
  \left.\kern-\nulldelimiterspace 
  #1 
  \littletaller 
  \right|_{#2} 
  }}
\newcommand{\littletaller}{\mathchoice{\vphantom{\big|}}{}{}{}}

\DeclareMathOperator {\dist} {dist}

\DeclareMathOperator{\F} {\mathcal{F}}

\begin{document}

\title{On Instability Properties of the Fractional Calder\'{o}n Problem}

\author{Hendrik Baers}
\address{Institute for Applied Mathematics, University of Bonn, Endenicher Allee 60, 53115 Bonn, Germany}
\email{hendrik.baers@uni-bonn.de}
\thanks{}

\author{Giovanni Covi}
\address{Department of Mathematics and Statistics, University of Helsinki, Pietari Kalmin katu 5 (Exactum), 00560 Helsinki, Finland}
\email{giovanni.covi@helsinki.fi}
\thanks{}

\author{Angkana Rüland}
\address{Institute for Applied Mathematics and Hausdorff Center for Mathematics, University of Bonn, Endenicher Allee 60, 53115 Bonn, Germany}
\email{rueland@uni-bonn.de}
\thanks{}



\date{\today}

\begin{abstract}
We prove exponential instability properties for the fractional Calder\'on problem and the conductivity formulation of the fractional Calder\'on problem in the regime of fractional powers $s\in (0,1)$. We particularly focus on two settings: First, we discuss instability properties in general domain geometries with scaling critical $L^{\frac{n}{2s}}$ potentials and constant background metrics. Secondly, we investigate instability properties in general geometries with $L^{\frac{n}{2s}}$ potentials and low regularity, variable coefficient, possibly anisotropic background metrics.
In both settings we make use of the methods introduced in \cite{KRS21} and we deduce strong compression estimates for the forward problem. In the first setting this is based on analytic smoothing estimates for a suitable comparison operator while in the second setting involving low regularity metrics this is based on an iterated compression gain. We thus generalize the results from \cite{RS18} to generic geometries and variable coefficients and further also discuss the setting of fractional conductivity equations. In particular, this proves that the logarithmic stability estimates for the fractional Calder\'on problem from \cite{RS20} are optimal.
\end{abstract}

\maketitle

\section{Introduction}
\label{sec:intro}

In this article, we study instability properties of recovering scaling critical $L^{\frac{n}{2s}}$-regular potentials for the fractional Calder\'on problem in general geometries, first for constant and secondly for low regularity, variable coefficient, possibly anisotropic background metrics. Moreover, we also provide analogous results for a fractional conductivity equation. Here and in the following we will always assume that $n\geq2$. Let us begin by outlining the set-up under investigation for constant coefficient background metrics. In this context, we investigate the stability properties associated with the recovery of a potential $q \in L^{\frac{n}{2s}}(\Omega)$ from the knowledge of the (generalized) Dirichlet-to-Neumann map 
\begin{align}
\label{eq:DtN}
\Lambda_q: \widetilde{H}^s(W) \rightarrow H^{-s}(W), \qquad f \mapsto \restr{\left( -\Delta \right)^s u}{W}.
\end{align}
Here, $s\in (0,1)$, $\Omega \subset \R^n$ denotes an open, bounded, non-empty, Lipschitz set, $W \subset \Omega_e:= \R^n \setminus \overline\Omega$ is an open, non-empty, Lipschitz set and the function $u \in H^s(\R^n)$ is a solution to
\begin{align}
\label{eq:nonlocal_Schroedinger}
\begin{cases}
\begin{alignedat}{2}
\left( (-\Delta)^s + q \right) u &= 0 \qquad &&\text{in } \Omega,\\
u &= f \qquad &&\text{in }  \Omega_e,
\end{alignedat}
\end{cases}
\end{align}
for $q\in L^{\frac{n}{2s}}(\Omega)$ satisfying the assumption (Aq) from below. In particular, the condition (Aq) guarantees the well-definedness of the linear, bounded operator $\Lambda_q$. For a definition of the function spaces we refer to Section \ref{sec:prelim} below.

In what follows, building on the results from \cite{CS07,ST10}, we interpret the equation \eqref{eq:nonlocal_Schroedinger} through its Caffarelli-Silvestre extension in the form introduced in \cite{CS07}. For $f\in \widetilde{H}^s(W)$, we set
\begin{align*}
(-\Delta)^s u := -c_s\lim\limits_{x_{n+1}\rightarrow 0} x_{n+1}^{1-2s}\p_{n+1} \tilde{u},
\end{align*}
for a suitable constant $c_s > 0$, where $\tilde{u} \in \dot{H}^{1}(\R^{n+1}_+, x_{n+1}^{1-2s})$  is a solution to 
\begin{align}
\label{eq:CS}
\begin{cases}
\begin{alignedat}{2}
\nabla \cdot x_{n+1}^{1-2s} \nabla \tilde{u} & = 0 \qquad &&\mbox{in } \R^{n+1}_+,\\
\tilde{u} & = f \qquad &&\mbox{in } \Omega_e \times \{0\},\\
c_s\lim\limits_{x_{n+1} \rightarrow 0} x_{n+1}^{1-2s}\p_{n+1} \tilde{u} & =  qu \qquad &&\mbox{in } \Omega,
\end{alignedat}
\end{cases}
\end{align}
with the boundary trace $u(x):= \tilde{u}(x,0) \in H^s(\R^n)$ being well-defined.
We refer to Section \ref{sec:prelim} for a weak formulation of \eqref{eq:CS} which guarantees well-posedness of the problem \eqref{eq:CS} in case that $q\in L^{\frac{n}{2s}}(\Omega)$ satisfies the assumption (Aq).
Here and throughout the article, we will assume the following condition for the potentials $q\in L^{\frac{n}{2s}}(\Omega)$:
\begin{itemize}
\item[(Aq)] There exists a radius $r_0>0$ such that for all $\tilde{q} \in B_{r_0}^{ L^{\frac{n}{2s}}(\Omega)}(q)$ zero is not a Dirichlet eigenvalue of the problem \eqref{eq:nonlocal_Schroedinger}, \eqref{eq:VariableCoeffCS} or \eqref{eq:conduc_main}, respectively.
\end{itemize}
As discussed in Section \ref{sec:prelim}, this can, for instance, be guaranteed if $q>0$ or if $\|q\|_{L^{\infty}(\Omega)}\leq \lambda_1$ is sufficiently small, where $\lambda_1>0$ denotes the first Dirichlet eigenvalue of the operator with vanishing potential.

We remark that alternative, equivalent interpretations of the fractional Schrödinger equation \eqref{eq:nonlocal_Schroedinger} are given in terms of suitable properties of its integral kernel or by spectral interpretations \cite{ST10} (c.f. Section \ref{sec:prelim} for details on our definition of the fractional operator \eqref{eq:nonlocal_Schroedinger}).

\subsection{The main results}
\label{sec:main_result}
Motivated by the classical, local Calder\'on problem, the study of the fractional Calder\'on problem has been the object of an intensive research activity, in which uniqueness, stability and reconstruction results have been deduced (see Section \ref{sec:literature} below). In this article, we seek to provide rather general instability results for the fractional Calder\'on problem as well as the fractional conductivity equation. 

\subsubsection{The fractional Calder\'on problem in general geometries with constant background metric}

We begin by proving instability for the fractional Calder\'on problem in rather general geometries in which no symmetries are present in general.

\begin{theorem}\label{ThmInstability}
Let $s\in (0,1)$. Let $\Omega \subset \R^n$ be open, non-empty, bounded and Lipschitz, $\Omega' \subset \Omega$ be compact, non-empty, Lipschitz and $W \subset \Omega_e$ open, non-empty, Lipschitz. Let $\delta>0$. Assume that $\bar{q} \in L^{\frac{n}{2s}}(\Omega)$ with $\supp(\bar{q}) \subset \Omega'$ satisfies the condition from (Aq) with radius $r_0>0$ and define $K := \{ q \in L^{\frac{n}{2s}} (\Omega) : q \in  \bar{q}+ B_{r_0}^{W_0^{\delta,\frac{n}{2s}}(\Omega')}(0) \}$. Let $\Lambda_q: \widetilde{H}^s(W) \to H^{-s}(W)$ be the Dirichlet-to-Neumann map as in \eqref{eq:DtN} above.

Then there exists $c>0$ such that for any $\varepsilon \in (0, r_0)$, there exist $q_1, q_2 \in K $ with
\begin{align*}
\Vert \Lambda_{q_1} - \Lambda_{q_2} \Vert_{\widetilde{H}^s(W) \to H^{-s}(W)} &\leq \exp\left(-c \varepsilon^{-\frac{1}{\delta(2+\frac{1}{n})}}\right),\\
\Vert q_1 - q_2 \Vert_{L^{\frac{n}{2s}}(\Omega)} &\geq \varepsilon.
\end{align*} 
In particular, if one has the stability property
\begin{align*}
\Vert q_1 - q_2 \Vert_{L^{\frac{n}{2s}}(\Omega)} \leq \omega(\Vert \Lambda_{q_1} -\Lambda_{q_2} \Vert_{\widetilde{H}^s(W) \to H^{-s}(W)} ), \qquad q_1,q_2 \in K,
\end{align*}
then necessarily $\omega(t) \gtrsim \vert \log t \vert^{-\delta(2+\frac{1}{n})}$ for $t$ small (the notation "$\lesssim$" will be introduced in Section \ref{ssec:prel-Notation}).
\end{theorem}

Let us comment on this result and the main underlying ideas: As in \cite{RS18} Theorem \ref{ThmInstability} shows the severe, logarithmic instability of the fractional Calder\'on problem, illustrating that (up to the precise exponents) stability estimates such as in \cite{RS20} are indeed optimal. However, as in \cite{KRS21} and contrary to \cite{RS18}, in the present article, we do not impose any symmetry assumptions on the underlying domain geometry. Moreover, we only require $L^{\frac{n}{2s}}$ regularity for the potentials. We emphasize that due to the comparably low regularity of the potential, the forward map is not analytically regularizing under these conditions.
In order to nevertheless deduce the claimed instability properties, we rely on the general framework laid out in \cite{KRS21} (building on \cite{DCR03} and \cite{M01}), combining ideas of analytic regularization for a suitable, analytically regularizing comparison problem and transferring the resulting compression estimates encoded in associated entropy number bounds to the original inverse problem. To this end, an important ingredient in the first part of the article consists in the derivation of strong, analytic regularizing properties for a forward comparison operator with respect to the tangential variables in open subsets of $\Omega$. We remark that due to the presence of the Muckenhoupt weight $x_{n+1}^{1-2s}$ the strong local regularization in tangential directions is optimal in classical (non-weighted) function spaces, as at the boundary of $\Omega$ one cannot hope for analytic regularity in the classical sense due to the presence of the characteristic $\dist(x,\partial \Omega)^s$ behaviour \cite{R16}. In order to infer the desired regularity results, we use the heat kernel characterization of the fractional Laplacian and transfer heat bounds to bounds for our nonlocal problem. We refer to \cite{FMMS22, FMMS23} for an earlier study of analyticity properties of the fractional Laplacian in polygons and polyhedra which is used to deduce exponential convergence rates for certain finite element methods in \cite{FMMS23a}. In the context of interior analyticity estimates, we view our heat kernel approach as a short, self-contained alternative to these arguments.

\subsubsection{The variable coefficient fractional Calderón problem}

We next turn to the instability of the fractional Calder\'on problem involving low regularity metrics. We follow the strategy from \cite[Section 5]{KRS21} which builds on an iterative gain of compression by a minimal amount of regularization. Relying on the variable coefficient Caffarelli-Silvestre extension from \cite{ST10} (see also \cite{ATEW18}), in this part of the article, we interpret the equation \eqref{eq:nonlocal_Schroedinger} through its extension, i.e., for $f\in \widetilde{H}^s(W)$, we set
\begin{align*}
(-\nabla \cdot a \nabla)^s u := -c_{s}\lim\limits_{x_{n+1}\rightarrow 0} x_{n+1}^{1-2s}\p_{n+1} \tilde{u},
\end{align*}
for a suitable constant $c_s > 0$. Here $\tilde{u} \in \dot{H}^{1}(\R^{n+1}_+, x_{n+1}^{1-2s})$ is a weak solution to 
\begin{align}\label{eq:VariableCoeffCS}
\begin{cases}
\begin{alignedat}{2}
\nabla \cdot x_{n+1}^{1-2s} \tilde{a} \nabla \tilde{u} & = 0 \qquad &&\mbox{in } \R^{n+1}_+,\\
\tilde{u} & = f \qquad &&\mbox{in } \Omega_e \times \{0\},\\
c_s\lim\limits_{x_{n+1} \rightarrow 0} x_{n+1}^{1-2s}\p_{n+1} \tilde{u} & =  qu \qquad &&\mbox{in } \Omega,
\end{alignedat}
\end{cases}
\end{align}
with the boundary trace $u(x):= \tilde{u}(x,0) \in H^s(\R^n)$ being well-defined. The metric $a \in L^{\infty}(\R^n, \R^{n\times n})$ is symmetric and uniformly elliptic, and $\tilde{a}:= \begin{pmatrix} a & 0 \\ 0 & 1 \end{pmatrix}$. 
Assuming that the condition (Aq) holds, we consider the associated Dirichlet-to-Neumann map
\begin{align}
\label{eq:DtNv}
\Lambda_{a,q}: \widetilde{H}^s(W) \rightarrow H^{-s}(W), \ f \mapsto \restr{ \left( -c_s \lim\limits_{x_{n+1} \rightarrow 0} x_{n+1}^{1-2s}\p_{n+1} \tilde{u}(x,x_{n+1}) \right)}{W}.
\end{align}
For more details on this formulation see also Section \ref{ssec:prel-VariableCoeffFrLapl}.

In this setting without \textit{any} analytic regularization, we prove that the following analogue of Theorem \ref{ThmInstability} holds.

\begin{theorem}\label{thm:low_reg_instab}
Let $s\in (0,1)$. Let $\Omega \subset \R^n$ be open, non-empty, bounded and Lipschitz, $\Omega' \subset \Omega$ be compact, non-empty, Lipschitz and $W \subset \Omega_e$ open, non-empty, Lipschitz. Let $\delta>0$. Assume one of the following:
\begin{enumerate}[(i)]
\item $\bar{q} \in L^{\frac{n}{2s}}(\Omega)$ with $\supp(\bar{q}) \subset \Omega'$ satisfies the condition from (Aq) with radius $r_0>0$.
\item $\bar{q} \in L^{p}(\Omega)$ with $p> \frac{n}{2s}$, or $p=\frac{n}{2s}$ and $\|\bar{q}\|_{L^p(\Omega)}\leq \theta$ for a sufficiently small constant $\theta>0$, such that it satisfies the condition from (Aq) with radius $r_0>0$.
\end{enumerate}
Define $K := \{ q \in L^{p} (\Omega) : q \in  \bar{q}+ B_{r_0}^{W_0^{\delta,p}(\Omega')}(0) \}$ for $p\geq\frac{n}{2s}$ and let $\Lambda_{a,q}: \widetilde{H}^s(W) \to H^{-s}(W)$ be the Dirichlet-to-Neumann map as in \eqref{eq:DtNv} above.

There exists $c>0$ such that for any $\varepsilon \in (0, r_0)$, there exist $q_1, q_2 \in K$ with
\begin{align*}
\Vert \Lambda_{a,q_1} - \Lambda_{a,q_2} \Vert_{\widetilde{H}^s(W) \to H^{-s}(W)} &\leq \exp\left(-c \varepsilon^{-\frac{1}{\delta(2+\frac{5}{n})}}\right),\\
\Vert q_1 - q_2 \Vert_{L^{p}(\Omega)} &\geq \varepsilon.
\end{align*} 
In particular, if one has the stability property
\begin{align*}
\Vert q_1 - q_2 \Vert_{L^{\frac{n}{2s}}(\Omega)} \leq \omega(\Vert \Lambda_{a,q_1} -\Lambda_{a, q_2} \Vert_{\widetilde{H}^s(W) \to H^{-s}(W)} ), \qquad q_1,q_2 \in K,
\end{align*}
then necessarily $\omega(t) \gtrsim \vert \log t \vert^{-\delta(2+\frac{5}{n})}$ for $t$ small.
\end{theorem}

Building on the ideas from \cite[Chapter 5]{KRS21} in the context of inverse problems and similar ideas from \cite{BH03} in the context of numerical approximation, we will use a minimal amount of regularization for the forward map by exploiting a suitably weighted Caccioppoli estimate in order to infer the desired compression estimates. Further, we refer to \cite{KM19} and the references therein for related ideas in the context of numerical approximation properties of the constant coefficient fractional Laplacian. These rely on the framework from \cite{BH03} and give rise to closely related compression properties. The central use of Caccioppoli estimates in these general compression bounds are also at the origin of the assumptions (i), (ii) from above. Indeed, these assumptions allow us to absorb the arising boundary contributions into the bulk estimates (see Proposition \ref{PropCacciopoli} and its proof). In our instability argument already a minimal amount of regularization yields sufficient compactness in weighted Sobolev spaces in order to infer a corresponding compression gain in the entropy numbers. We emphasize that in this low regularity setting and contrary to the argument leading to Theorem \ref{ThmInstability}, we also allow for background potentials $\bar{q}$ which are not compactly supported in $\Omega$ and that only the difference $q - \bar{q}$ is required to have compact support. In particular, this will allow us to transfer the instability results for the (variable) coefficient fractional Calder\'on problem to its conductivity version (see Sections \ref{sec:fractional_conductivity} and \ref{sec:low_reg_proof2} below). We remark that, clearly, the framework in Theorem \ref{thm:low_reg_instab} allows to include the setting from Theorem \ref{ThmInstability}. However, due to the different mechanisms (analytic smoothing vs. a minimal amount of smoothing for the forward operator) leading to the respective instability result, we consider both mechanisms of interest in their own right.

\subsubsection{The fractional conductivity equation}
\label{sec:fractional_conductivity}
In this article we also seek to prove an instability result in the case of the fractional Calder\'on problem for the following fractional conductivity equation 
\begin{align}
\label{eq:conduc_main}
\begin{cases}
\begin{alignedat}{2}
\left( (\nabla \cdot)^s (\Theta \cdot\nabla^s) + q \right) u &= 0 \qquad &&\text{in } \Omega,\\
u &= f \qquad &&\text{in }  \Omega_e.
\end{alignedat}
\end{cases}
\end{align}
The operators $(\nabla\cdot)^s, \nabla^s$, which will be defined in detail in Section \ref{ssec:prel-FracCondEq}, are the fractional divergence and gradient. They were introduced in \cite{DGLZ12} as fundamental operators of nonlocal vector calculus, and were then further studied in many subsequent works, among which \cite{C20, C20a, CRZ22, C22}. They can be thought of as fractional counterparts to the usual divergence and gradient operators, and enjoy the familiar properties $(\nabla^s)^* = (\nabla\cdot)^s$ and $(\nabla\cdot)^s\nabla^s = (-\Delta)^s$. The interaction matrix $\Theta$ is defined as $\Theta(x,y):=\gamma(x)^{1/2}\gamma(y)^{1/2}I$, where $I \in \R^{n \times n}$ is the identity matrix and the scalar, positive function $\gamma\in C^\infty(\R^n)$ is the conductivity. In the framework of this problem, in the present article, we discuss the following instability result.

\begin{theorem}\label{ThmInstability_other}
Let $s\in (0,1)$. Let $\Omega \subset \R^n$ be open, non-empty, bounded and Lipschitz, $\Omega' \subset \Omega$ be compact, non-empty, Lipschitz and $W \subset \Omega_e$ open, non-empty, Lipschitz. Let $\gamma\in C^\infty(\R^n)$ satisfy $\gamma^{1/2}-1 \in C^\infty_c(\Omega)$ and $\gamma\geq \underline\gamma$ for some $\underline\gamma \in (0,\infty)$. Let $\delta \in (0,1)$. Assume that $\bar{q} \in L^{p}(\Omega)$ with $p> \frac{n}{2s}$ satisfies the condition from (Aq) with radius $r_0>0$, and define $K := \{ q \in L^{p} (\Omega) : q \in  \bar{q}+ B_{r_0}^{W_0^{\delta,p}(\Omega')}(0) \}$.  Let $\Lambda_{\gamma,q}: \widetilde{H}^s(W) \to H^{-s}(W)$ be the Dirichlet-to-Neumann map defined by
$$ \Lambda_{\gamma,q} : f \mapsto  (\nabla \cdot)^s (\Theta \cdot\nabla^s)u_f|_W .$$

There exists $c>0$ such that for any $\varepsilon \in (0, r_0)$, there exist $q_1, q_2 \in K$ with
\begin{align*}
\Vert \Lambda_{\gamma,q_1} - \Lambda_{\gamma,q_2} \Vert_{\widetilde{H}^s(W) \to H^{-s}(W)} &\leq \exp\left(-c \varepsilon^{-\frac{1}{\delta(2+\frac{5}{n})}}\right),\\
\Vert q_1 - q_2 \Vert_{L^{p}(\Omega)} &\geq \varepsilon.
\end{align*} 
\end{theorem}

Relying on ideas first introduced in \cite{C20}, our argument for this instability result for the fractional conductivity problem makes use of the fractional Liouville transform to reduce the fractional conductivity equation to a fractional Schrödinger equation. On this level, we may invoke the results from the previous two sections. We highlight that in this context, due to the nonlocality of the fractional Liouville transform, we cannot avoid dealing with background potentials $\bar{q}$ with non-compact support.

\subsubsection{Single measurement instability results}
Last but not least, as corollaries of our above results, we will consider instability properties for single measurement reconstruction problems. These results will be deduced as simple consequences of our main results in Theorems \ref{ThmInstability} and \ref{ThmInstability_other}.\\
For the single measurement fractional Calder\'{o}n problem a similar logarithmic stability result as for the infinitely many measurements case was proven in \cite{R21}. We complement this result with a lower bound, once again showing that the logarithmic stability result is indeed optimal.

\begin{corollary}\label{CorSingleMeasFractCald}
Let $s\in (0,1)$. Let $\Omega \subset \R^n$ be open, non-empty, bounded and Lipschitz, $\Omega' \subset \Omega$ compact, non-empty, Lipschitz and $W \subset \Omega_e$ open, non-empty, Lipschitz. Let $\delta>0$. Assume that $\bar{q} \in L^{\frac{n}{2s}}(\Omega)$ with $\supp(\bar{q}) \subset \Omega'$ satisfies the condition from (Aq) with radius $r_0>0$, and define $K := \{ q \in L^{\frac{n}{2s}} (\Omega) : q \in  \bar{q}+ B_{r_0}^{W_0^{\delta,\frac{n}{2s}}(\Omega')}(0) \}$. Let $f \in \widetilde{H}^s(W) \setminus \{0\}$ and let $\Lambda_q: \widetilde{H}^s(W) \to H^{-s}(W)$ be the Dirichlet-to-Neumann map as in \eqref{eq:DtN}.

There exists $c>0$ such that for any $\varepsilon \in (0, r_0)$, there exist $q_1, q_2 \in K$ with
\begin{align*}
\Vert \Lambda_{q_1}(f) - \Lambda_{q_2}(f) \Vert_{H^{-s}(W)} &\leq \exp\left(-c \varepsilon^{-\frac{1}{\delta(2+\frac{1}{n})}}\right),\\
\Vert q_1 - q_2 \Vert_{L^{\frac{n}{2s}}(\Omega)} &\geq \varepsilon.
\end{align*} 
In particular, if one has the stability property
\begin{align*}
\Vert q_1 - q_2 \Vert_{L^{\frac{n}{2s}}(\Omega)} \leq \omega(\Vert \Lambda_{q_1}(f) -\Lambda_{q_2}(f) \Vert_{H^{-s}(W)} ), \qquad q_1,q_2 \in K,
\end{align*}
then necessarily $\omega(t) \gtrsim \vert \log t \vert^{-\delta(2+\frac{1}{n})}$ for $t$ small.
\end{corollary}

\begin{remark}
We remark that with the same argument, using Theorem \ref{thm:low_reg_instab}, we can derive the analogous single measurement instability in the low regularity anisotropic setting.
\end{remark}

Moreover, we deduce an instability result for the single measurement reconstruction problem for the fractional conductivity equation as in Section \ref{sec:fractional_conductivity}.

\begin{corollary}\label{CorSingleMeasFractCond}
Let $s\in (0,1)$. Let $\Omega \subset \R^n$ be open, non-empty, bounded and Lipschitz, $\Omega' \subset \Omega$ compact, non-empty, Lipschitz and $W \subset \Omega_e$ open, non-empty, Lipschitz. Let $\gamma\in C^\infty(\R^n)$ satisfy $\gamma^{1/2}-1 \in C^\infty_c(\Omega)$ and $\gamma\geq \underline\gamma$ for some $\underline\gamma \in (0,\infty)$. Let $\delta>0$. Assume that $\bar{q} \in L^{p}(\Omega)$ with $p> \frac{n}{2s}$ satisfies the condition from (Aq) with radius $r_0>0$, and define $K := \{ q \in L^{p} (\Omega) : q \in  \bar{q}+ B_{r_0}^{W_0^{\delta,p}(\Omega')}(0) \}$. Let $f \in \widetilde{H}^s(W)$ and let $\Lambda_{\gamma,q}: \widetilde{H}^s(W) \to H^{-s}(W)$ be the Dirichlet-to-Neumann map as in Theorem \ref{ThmInstability_other}.\\
There exists $c>0$ such that for any $\varepsilon \in (0, r_0)$, there exist $q_1, q_2 \in K$ with
\begin{align*}
\Vert \Lambda_{\gamma,q_1}(f) - \Lambda_{\gamma,q_2}(f) \Vert_{H^{-s}(W)} &\leq \exp\left(-c \varepsilon^{-\frac{1}{\delta(2+\frac{5}{n})}}\right),\\
\Vert q_1 - q_2 \Vert_{L^{p}(\Omega)} &\geq \varepsilon.
\end{align*}
\end{corollary}

Making use of the Liouville transform and the single measurement stability result from \cite{R21}, we complement the single measurement instability estimate for the fractional conductivity by a corresponding stability result. As in \cite{R21}, however, we here impose stronger regularity assumptions on the potentials than for the instability result.

\begin{proposition}\label{PropSingleMeasStabFractCond}
Let $s\in (0,1)$. Let $\Omega \subset \R^n$ be open, non-empty, $C^2$-regular and bounded, $\Omega' \subset \Omega$ compact, non-empty, $C^2$-regular and $W \subset \Omega_e$ open, non-empty and $C^2$-regular. Assume that $\overline{\Omega} \cap \overline{W} = \emptyset$. Let $\gamma\in C^\infty(\R^n)$ satisfy $\gamma^{1/2}-1 \in C^\infty_c(\Omega)$ and $\gamma\geq \underline\gamma$ for some $\underline\gamma \in (0,\infty)$. Let $\Lambda_{\gamma,q}: \widetilde{H}^s(W) \to H^{-s}(W)$ be the Dirichlet-to-Neumann map as in Theorem \ref{ThmInstability_other}.\\
Let $f \in \widetilde{H}^{s+\varepsilon}(W) \setminus \{0\}$ for some $\varepsilon>0$, and let $F>0$ be such that
\begin{align*}
\frac{\Vert f \Vert_{H^s(W)}}{\Vert f \Vert_{L^2(W)}} \leq F.
\end{align*}
Assume that $q_1,q_2 \in C^{0,s}(\Omega)$ with $\supp(q_2-q_1) \subset \Omega'$ satisfy the bound
\begin{align*}
\Vert q_j \Vert_{C^{0,s}(\Omega)} \leq E < \infty \quad \text{for } j \in \{1,2\}.
\end{align*}
Then there exists a modulus of continuity $\omega(t)$ with $\omega(t) \leq c \vert \log(ct) \vert^{-\mu}$ for $t\in (0,1]$ and some constants $\mu>0$ and $c>1$ such that
\begin{align*}
\Vert q_1 - q_2 \Vert_{L^\infty(\Omega)} \leq \omega \left( \Vert \Lambda_{\gamma,q_1}(f) - \Lambda_{\gamma,q_2}(f) \Vert_{H^{-s}(W)} \right), 
\end{align*}
where $\mu$ and $c$ only depend on $\Omega$, $W$, $s$, $E$, $F$, $\gamma$, $n$, $\Vert f \Vert_{H^{s+\varepsilon}(W)}$, $\dist(\Omega', \partial\Omega)$.
\end{proposition}

\subsection{Relation to the literature}
\label{sec:literature}
The results of Theorem \ref{ThmInstability} - \ref{ThmInstability_other} and Corollaries \ref{CorSingleMeasFractCald}, \ref{CorSingleMeasFractCond} should be viewed in the context of the strong activity on fractional Calder\'on type problems: The fractional Calder\'on problem was introduced in the seminal work \cite{GSU20} in which partial data uniqueness results for $L^{\infty}$ regular potentials were derived. Building on this, low regularity partial data uniqueness and optimal stability results \cite{RS20, RS18}, uniqueness in the presence of anisotropic background metrics \cite{GLX17}, single measurement reconstruction and stability properties \cite{GRSU20, R21} as well as reconstruction based on monotonicity methods \cite{HL20, HL20a} and finite dimensional versions \cite{RS19} were deduced. In \cite{C20} a conductivity formulation and a Liouville reduction was introduced. This was further extended in \cite{C20a,CMRU20,C22,CRZ22,RZ24}.
We also point to \cite{RS20a,LLR20,CLR20,BGU21,L21,L23,CdHS22} for a non-exhaustive list of further results on nonlinear and on related fractional equations as well as to \cite{GU21, CGRU23, LLU23,Lin23} for works connecting the local and fractional Calder\'on problems and to \cite{F24, FGKU21, QU22, C23, CO23,R23} for novel results on the reconstruction of metrics in fractional Calder\'on type problems.
Building on Mandache's \cite{M01} and Di Cristo-Rondi's \cite{DCR03} work, in \cite{RS18} a first instability analysis for the fractional Calder\'on problem was carried out -- showing that contrary to the very strong uniqueness properties, the modulus of continuity does \emph{not} improve substantially compared to the classical Calder\'on problem. As in \cite{M01} and \cite{DCR03}, the results of \cite{RS18} however relied on the use of strong symmetries in the operators and domains.
In \cite{KRS21} a systematic study of instability mechanisms was taken up which allowed to robustify the previous instability mechanisms for a large class of inverse problems. It is this framework which also forms the foundation for our current analysis. Here we are confronted with novel challenges arising due to nonlocality (e.g., in having to consider non-constant background potentials).

\subsection{Outline of the article}
\label{sec:outline}
The remainder of the article is structured as follows: In Section \ref{sec:prelim} we first discuss preliminary results on the precise set-up of our nonlocal equations and some fundamental quantities for our instability framework. In Section \ref{sec:instab} we then turn to proving our instability results for the constant coefficient fractional Calder\'on problem. In particular, this includes a regularity analysis of a comparison operator in Section \ref{sec:reg}. Next, in Section \ref{sec:low_reg} we present the argument for the compression and instability estimates in the presence of low regularity background metrics. Here a major ingredient is a regularity gain obtained through a Caccioppoli estimate together with entropy bounds in weighted $L^2$ based function spaces. In Section \ref{SectionProofCondEq} we focus on our last model example, the fractional conductivity equation, for which we also deduce exponential instability properties. After that, in Section \ref{sec:single_meas}, we prove the instability results for the single measurement reconstruction problem. Last but not least, the main body of our text is complemented by an appendix in which we collect (well-known) analyticity estimates for the heat equation for the convenience of the reader.

\section{Preliminaries}
\label{sec:prelim}

We start with some preliminaries about the function spaces and operators which will appear in our arguments.

\subsection{Entropy numbers}\label{ssec:prel-entropynumbers}
We recall some definitions and properties of entropy numbers needed for our analysis. Our main reference here is \cite{KRS21}, which builds on \cite{EdmundsTriebel1996} and \cite{Kolmogorov1959}.

\begin{definition}
Let $X$, $Y$ be Banach spaces and let $A: X \to Y$ be a bounded linear operator. For any $k\geq1$, the $k$-th entropy number $e_k(A)$ of $A$ is defined by
\begin{align*}
e_k(A) := \inf \left\{ \varepsilon>0: A(B_1^X) \subset \bigcup_{j=1}^{2^{k-1}} (y_j + \varepsilon B_1^Y) \text{ for some } y_1,\dots,y_{2^{k-1}} \in Y \right\}.
\end{align*} 
\end{definition}

The following two properties are relevant for our analysis.

\begin{lemma}
Let $X$, $Y$, $Z$ be Banach spaces and let $A: X \to Y$, $B: Y \to Z$ be bounded linear operators with $A \not\equiv 0$. Then for all $j,k \geq 1$ it holds that
\begin{align*}
e_{j+k-1} (B \circ A) \leq e_j(B) e_k(A).
\end{align*}
\end{lemma}

\begin{lemma}[{\cite[Lemma 3.9]{KRS21}}]\label{Lemma3.9[KRS21]}
Let $A: X \to Y$ be a compact operator between separable Hilbert spaces. Let $\sigma_k(A)$ denote the singular values of $A$. Then, for $\mu>0$ it holds that
\begin{align*}
\sigma_k(A) \lesssim e^{-ck^\mu} \text{ for some } c>0 \quad \Longleftrightarrow \quad e_k(A) \lesssim e^{-\tilde{c}k^{\frac{\mu}{1+\mu}}} \text{ for some } \tilde{c}>0.
\end{align*}
\end{lemma}

\subsection{Sobolev spaces}\label{ssec:prel-sobo}
We introduce the relevant function spaces for our analysis. Let $s\in\R$. The inhomogeneous fractional Sobolev space $H^s(\R^n)$ is defined as
$$ H^s(\R^n):= \{ u\in \mathcal S'(\R^n) : \|u\|_{H^s(\R^n)} := \| (1+|\xi|^2)^{s/2}\hat u(\xi)\|_{L^2(\R^n)} <\infty \}, $$ while for the homogeneous fractional Sobolev space $\dot H^s(\R^n)$ we let $$  \dot H^s(\R^n):= \{ u\in \mathcal S'(\R^n) : \|u\|_{\dot H^s(\R^n)} := \| |\xi|^s\hat u(\xi)\|_{L^2(\R^n)} <\infty \}, $$ where $\mathcal S'(\R^n)$ indicates the set of tempered distributions on $\R^n$. Let now $\Omega\subset\R^n$ be open, bounded and Lipschitz. In this case we define the inhomogeneous fractional Sobolev space $H^s(\Omega)$ on the bounded domain $\Omega$ as $$ H^s(\Omega):=\{u|_\Omega: \ u\in H^s(\R^n)\}, $$ to which we associate the quotient norm $\|u\|_{H^s(\Omega)} := \inf\{ \|U\|_{H^s(\R^n)} : U|_\Omega=u \}. $ We also define the spaces
$$\widetilde H^s(\Omega):= \overline{C^\infty_c(\Omega)}^{H^s(\R^n)}, \qquad H^s_{\overline\Omega} := \{ u\in H^s(\R^n) : \supp( u)\subseteq \overline\Omega\},$$ and observe that $\widetilde H^s(\Omega) = H^s_{\overline\Omega}$ holds for Lipschitz $\Omega$ (see \cite[Theorem 3.29]{McLean}). The following properties show how the two kinds of fractional Sobolev spaces on bounded domains are related to each other by duality: $$ (\widetilde H^s(\Omega))^* = H^{-s}(\Omega), \qquad (H^s(\Omega))^* = \widetilde H^{-s}(\Omega).$$ It is possible to analogously define the homogeneous kind of fractional Sobolev spaces on bounded sets.

We also recall the Sobolev-Slobodeckij spaces for $\Omega\subset\R^n$ open, $s\in(0,1)$ and $1 \leq p < \infty$
\begin{align*}
W^{s,p}(\Omega) := \{u \in L^p(\Omega): \Vert u \Vert_{W^{s,p}(\Omega)} < \infty \},
\end{align*}
where
\begin{align*}
\Vert u \Vert_{W^{s,p}(\Omega)} := \Vert u \Vert_{L^p(\Omega)} + \left( \int_\Omega \int_\Omega \frac{\vert u(x) - u(y) \vert^p}{\vert x-y \vert^{ps+n}} dx dy \right)^{1/p}.
\end{align*}
By $W_0^{s,p}(\Omega)$ we will denote the closure of $C_c^\infty(\Omega)$ in $W^{s,p}(\R^n)$.

Furthermore, for $s\in(0,1)$ and $\widetilde{\Omega}\subseteq\R^{n+1}_+ := \{ (x,x_{n+1}) \in \R^{n+1}: x_{n+1}>0 \}$ open, we define the following weighted inhomogeneous and homogeneous Sobolev spaces 
\begin{align*}
H^1(\widetilde{\Omega}, x_{n+1}^{1-2s}) &:= \{ \tilde{u}:\widetilde{\Omega} \rightarrow \R  : \| x_{n+1}^{\frac{1-2s}{2}} \tilde{u} \|_{L^2(\widetilde{\Omega})} + \| x_{n+1}^{\frac{1-2s}{2}}\nabla \tilde{u} \|_{L^2(\widetilde{\Omega})} <\infty  \}, \\
\dot H^1(\widetilde{\Omega}, x_{n+1}^{1-2s}) &:= \{ \tilde{u}:\widetilde{\Omega} \rightarrow \R : \| x_{n+1}^{\frac{1-2s}{2}}\nabla \tilde{u} \|_{L^2(\widetilde{\Omega})} <\infty  \}, \\
\dot{H}^{1}_0(\widetilde{\Omega}, x_{n+1}^{1-2s})&:=\{ \tilde{v} \in \dot{H}^{1}(\widetilde{\Omega}, x_{n+1}^{1-2s}): \ \tilde{v}|_{\p\widetilde{\Omega}}=0 \},
\end{align*}
where the trace is well-defined, see for example \cite[Lemma 4.4]{RS20}.

Additionally, we fix the following notation. For $\Omega \subset \R^n$ open, Lipschitz and $R>0$ we denote $Q_{\Omega,R} := \Omega \times (0,R) \subset \R^{n+1}_+$ and $\Gamma_{\Omega,R} := \Omega \times [0,R) \subset \overline{\R^{n+1}_+}$. Moreover, we define (with a slight abuse of notation) the space $H_0^1(\Gamma_{\Omega,R}, x_{n+1}^{1-2s})$ and its homogeneous counterpart by
\begin{align*}
H_0^1(\Gamma_{\Omega,R},x_{n+1}^{1-2s}) := \{ \tilde{u} \in H^1(Q_{\Omega,R}, x_{n+1}^{1-2s}): \tilde{u} = 0 \text{ on } \left( \partial\Omega \times(0,R) \right) \cup \left( \Omega \times \{R\} \right) \}
\end{align*}
and
\begin{align*}
\dot{H}_0^1(\Gamma_{\Omega,R},x_{n+1}^{1-2s}) := \{ \tilde{u} \in \dot{H}^1(Q_{\Omega,R}, x_{n+1}^{1-2s}): \tilde{u} = 0 \text{ on } \left( \partial\Omega \times(0,R) \right) \cup \left( \Omega \times \{R\} \right) \}.
\end{align*}
In Section \ref{sec:singular} we deduce singular value estimates for the associated embeddings of $ H^1(Q_{\Omega,R}, x_{n+1}^{1-2s}) \rightarrow  L^2(Q_{\Omega,R}, x_{n+1}^{1-2s})$.

\subsection{Gevrey spaces}\label{ssec:prel-gev} 
In what follows below, we will make use of the class of Gevrey regular functions. In contrast to \cite{KRS21}, we will directly work in the Euclidean setting and rely on an $L^2$ based notion of Gevrey classes. Let $\Omega \subset \R^n$ be an open, smooth domain, $\Omega' \Subset \Omega$, $\sigma\geq 1$ and $\rho>0$. We define the Gevrey-$\sigma$-spaces, $G^\sigma(\overline{\Omega'})$, by
\begin{align*}
G^\sigma(\overline{\Omega'}) := \{ u \in C^\infty(\overline{\Omega'}): \text{ there exist } C,R>0 \text{ s.t. } \Vert D^\alpha u \Vert_{L^\infty(\overline{\Omega'})} \leq C R^{\vert\alpha\vert} \vert \alpha \vert^{\sigma\vert\alpha\vert} \text{ for all } \alpha\in\N_0^n \}.
\end{align*}
In particular, the class $G^{1}(\overline{\Omega'})$ coincides with the analytic functions on $\overline{\Omega'}$. In contrast to the analytic functions, if $\sigma>1$, there are elements in $G^\sigma(\overline{\Omega'})$ with compact support. 

Moreover, following \cite{KRS21}, we define the Hilbert spaces $A^{\sigma, \rho}(\overline{\Omega'})  \subset L^2(\overline{\Omega'})$ for $1 \leq \sigma < \infty$ and $\rho > 0$ in the following way:
\begin{align*}
A^{\sigma, \rho}(\overline{\Omega'}) := \left\{ u \in C^{\infty}(\overline{\Omega'}): \ \|u\|_{A^{\sigma, \rho}(\overline{\Omega'})} < \infty \right\}, \ 
\|u\|_{A^{\sigma, \rho}(\overline{\Omega'})} := \left( \sum\limits_{\ell = 0}^{\infty} e^{-2 \ell \rho} e^{-2 \sigma \ell \log(\ell+1)} \|u\|_{H^{\ell}(\overline{\Omega'})}^2 \right)^{\frac{1}{2}}.
\end{align*}

\begin{proposition}\label{PropGevreySpaceIdentitiy}
Let $\sigma \geq 1$. The spaces $G^\sigma(\overline{\Omega'})$ and $A^{\sigma,\rho}(\overline{\Omega'})$ are related by 
\begin{align*}
G^\sigma(\overline{\Omega'}) = \bigcup_{\rho>0} A^{\sigma,\rho}(\overline{\Omega'}).
\end{align*}
\end{proposition}

\begin{proof}
Let $u \in G^\sigma(\overline{\Omega'})$, then
\begin{align*}
\Vert u \Vert_{A^{\sigma,\rho}(\overline{\Omega'})}^2 &= \sum_{\ell=0}^\infty e^{-2\ell\rho} e^{-2\sigma \ell \log(\ell+1)} \Vert u \Vert_{H^\ell (\overline{\Omega'})}^2 \leq \sum_{\ell=0}^\infty e^{-2\ell\rho} e^{-2\sigma \ell \log(\ell+1)} \ell C^2 R^{2\ell} \ell^{2\sigma\ell}\\
&\leq C^2 \sum_{\ell=0}^\infty \ell \exp \big( -2\ell\rho - 2\sigma \ell \log(\ell+1) + 2\ell\log(R) + 2\sigma\ell\log(\ell+1) \big) < \infty,
\end{align*}
if $\rho > \log(R)$, and thus $u \in A^{\sigma,\rho}(\overline{\Omega'})$.\\
On the other hand, let $u \in A^{\sigma,\rho}(\overline{\Omega'})$, i.e., $\Vert u \Vert_{A^{\sigma,\rho}(\overline{\Omega'})}^2 \leq M < \infty$, which gives by the definition of the $A^{\sigma,\rho}$-norm
\begin{align*}
e^{-2\ell\rho} e^{-2\sigma\ell \log(\ell+1)} \Vert u \Vert_{H^\ell(\overline{\Omega'})}^2 \leq M \qquad \text{for all } \ell\in\N.
\end{align*}
Let now $\alpha \in \N_0^n$ with $\vert\alpha\vert = k$. By Sobolev embedding,
\begin{align*}
\Vert D^\alpha u \Vert_{L^\infty(\overline{\Omega'})}^2 &\leq C \Vert u \Vert_{H^{k + \lceil \frac{n}{2} \rceil + 1}(\overline{\Omega'})}^2 \leq CM e^{2\rho(k + \lceil \frac{n}{2} \rceil + 1)} e^{2\sigma (k + \lceil \frac{n}{2} \rceil + 1) \log( k + \lceil \frac{n}{2} \rceil + 2 )}\\
&= CM e^{2\rho (\lceil \frac{n}{2} \rceil + 1)} e^{2\rho k} (k + \left\lceil \frac{n}{2} \right\rceil + 2)^{2\sigma (k + \lceil \frac{n}{2} \rceil + 1)} \leq \tilde{C}^2 \tilde{R}^{2k} k^{2\sigma k},
\end{align*}
where the last inequality holds since for fixed $m\in\N$ we have $(k+m)^{\sigma(k+m)} \leq C_1 R_1^k k^{\sigma k}$ for all $k\in\N$. This proves that $u \in G^\sigma(\overline{\Omega'})$.
\end{proof}

\begin{remark}\label{RemarkInclusionGevreySpace}
From the proof of Proposition \ref{PropGevreySpaceIdentitiy} we see that for $u \in G^\sigma(\overline{\Omega'})$ with
\begin{align*}
\Vert D^\alpha u \Vert_{L^\infty(\overline{\Omega'})} \leq CR^{\vert\alpha\vert} \vert\alpha\vert^{\sigma\vert\alpha\vert} \qquad \text{for all } \alpha\in\N_0^n
\end{align*}
the parameter $\rho$ such that $u\in A^{\sigma,\rho}(\overline{\Omega'})$ depends on $R$ but not on $C$.
\end{remark}

For $\sigma\geq 1, \ \rho, \ r>0$, we have the following size estimate of the inclusion $A^{\sigma, \rho}(\overline{\Omega'}) \subset H^{r}(\overline{\Omega'})$. A related estimate for closed smooth manifolds was deduced in \cite[Theorem 3.13]{KRS21}.

\begin{lemma}
\label{lem:entropy_numbers_embedding}
Let $\sigma\geq 1$, $\rho>0$ and $r>0$. Let $\Omega \subset \R^n$ be an open, Lipschitz domain and let $\Omega' \Subset \Omega$. Let $j: A^{\sigma, \rho}(\overline{\Omega'}) \rightarrow H^r(\overline{\Omega'})$ be the canonical inclusion and let $e_k(j)$ denote its entropy numbers. Then, there exist constants $c>0$, $\tilde{\rho}>0$ such that
\begin{align*}
e_k(j) \leq c e^{- \tilde{\rho} k^{\frac{1}{n\sigma +1 }}}.
\end{align*}
\end{lemma}

\begin{proof}
By definition, we have 
\begin{align*}
e_k(j) := \inf\left\{ \varepsilon>0: \ j(B^{{A}^{\sigma, \rho}}_1) \subset \bigcup\limits_{l=0}^{2^{k-1}} (y_l + B^{H^r}_{\varepsilon}) \right\},
\end{align*}
for $y_l \in H^r(\overline{\Omega'})$. Moreover, we recall that $e_k(i_{H^r \rightarrow H^{s}}) \sim k^{- \frac{r-s}{n}}$ and that, by definition, for any linear map $T$, it holds that $e_k(c T) = c e_k(T)$. As a consequence, since $B^{A^{\sigma, \rho}}_1 \subset e^{\ell \rho} e^{\sigma \ell \log(\ell+1)} B_1^{H^{\ell}} $ for all $\ell\in\N$, we have that
\begin{align*}
e_k(j) \leq \inf\limits_{\ell \in \N} e^{\ell \rho} e^{\sigma \ell \log(\ell+1)} e_k(i_{H^{\ell} \rightarrow H^{r}}) \leq c \inf\limits_{\ell \in \N} e^{\ell \rho} e^{\sigma \ell \log(\ell+1)} k^{\frac{-(\ell-r)}{n}}.
\end{align*}
It thus remains to optimize the expression from above. By choosing $\ell = k^\frac{1}{n\sigma +1}$ we obtain an upper bound which yields the claimed scaling for the entropy numbers.
\end{proof}

\subsection{The fractional Laplacian}\label{ssec:prel-frlap}
Let $s\in\mathbb R^+\setminus\mathbb Z$ and $u\in\mathcal S(\R^n)$, the set of Schwartz functions on $\R^n$. The fractional Laplacian of $u$ is defined as
$$(-\Delta)^s u := \mathcal{F}^{-1}(|\xi|^{2s}\hat u(\xi)).$$
Observe that $(-\Delta)^s$ is a continuous map from $\mathcal S(\R^n)$ to $L^\infty(\R^n)$, and thus the above definition can be extended by density to a continuous operator $(-\Delta)^s : H^r(\mathbb R^n)\rightarrow H^{r-2s}(\mathbb R^n)$ for all $r\in\mathbb R$. It is also possible to extend the definition of the fractional Laplacian to Sobolev spaces of negative exponent and to Sobolev spaces of integrability $p\neq 2$. Many equivalent definitions of the operator $(-\Delta)^s$ are given in \cite{Kw17} in the regime $s\in(0,1)$, such as the definition as the following singular integral
$$ (-\Delta)^su(x) := c_{n,s}\,PV\int_{\mathbb R^n}\frac{u(x)-u(y)}{|x-y|^{n+2s}}dy ,$$
with $c_{n,s}:= \frac{4^s \Gamma(n/2+s)}{\pi^{n/2}|\Gamma(-s)|}$ and where $PV$ denotes the principal value.

The fractional Laplacian can also be understood in terms of the Caffarelli-Silvestre extension \cite{CS07}. Since we will often use this perspective in what follows, we recall some details of this. Let $s\in (0,1)$, and consider the following extension problem relative to the function $u\in H^s(\R^n)$: 
\begin{align}\label{eq:CS-simple}
\begin{cases}
\begin{alignedat}{2}
\nabla \cdot x_{n+1}^{1-2s} \nabla \tilde{u} & = 0 \quad &&\mbox{in } \R^{n+1}_+,\\
\tilde{u} & = u \quad &&\mbox{in } \R^n \times \{0\}. 
\end{alignedat}
\end{cases}
\end{align}
This is naturally associated to the bilinear form 
\begin{align*}
B(v,w):= \int\limits_{\R^{n+1}_+} x_{n+1}^{1-2s} \nabla v \cdot \nabla w d(x,x_{n+1}),
\end{align*}
which is well-defined for $v,w \in \dot{H}^{1}(\R^{n+1}_+, x_{n+1}^{1-2s})$. It is known (\cite{CS07}) that problem \eqref{eq:CS-simple} admits a unique (weak) solution $\tilde u \in \dot{H}^{1}(\R^{n+1}_+, x_{n+1}^{1-2s})$, i.e., $\tilde u$ verifies $B(\tilde u,v)=0$ for all $v\in \dot{H}^{1}_0(\R^{n+1}_+, x_{n+1}^{1-2s})$, and $\tilde u|_{\R^n\times\{0\}}=u$. This function $\tilde u$ is called the \emph{Caffarelli-Silvestre extension of $u$}. As proved in \cite{CS07, ST10}, the fractional Laplacian $(-\Delta)^su$ of $u\in H^s(\R^n)$ can be expressed in terms of the (local) Dirichlet-to-Neumann map associated to problem \eqref{eq:CS-simple}: there exists a constant $c_s>0$ such that
\begin{align}\label{EqFractionalOperatorAsLimitFromCS}
(-\Delta)^s u(x) = -c_s \lim\limits_{x_{n+1} \rightarrow 0} x_{n+1}^{1-2s} \p_{n+1} \tilde{u}(x,x_{n+1}) \in {\dot{H}^{-s}(\R^n)}, 
\end{align} 
with both objects understood in their weak forms. In particular, the nonlocal problem
\begin{align*}
\begin{cases}
\begin{alignedat}{2}
(-\Delta)^su & = F \quad &&\mbox{in } \Omega,\\
u & = f \quad &&\mbox{in } \Omega_e,
\end{alignedat}
\end{cases}
\end{align*}
for $f\in H^s(\R^n), F\in H^{-s}(\Omega)$ can be reformulated locally in terms of the Caffarelli-Silvestre extension as
\begin{align}\label{eq:CS-frlap}
\begin{cases}
\begin{alignedat}{2}
\nabla \cdot x_{n+1}^{1-2s} \nabla \tilde{u} & = 0 \quad &&\mbox{in } \R^{n+1}_+,\\
-c_s \lim\limits_{x_{n+1} \rightarrow 0} x_{n+1}^{1-2s} \p_{n+1} \tilde{u} & = F \quad &&\mbox{in } \Omega, \\
\tilde{u} & = f \quad &&\mbox{in } \Omega_e \times \{0\}. 
\end{alignedat}
\end{cases}
\end{align}
We say that $\tilde u\in \dot H^1(\R^{n+1}_+,x_{n+1}^{1-2s})$ is a solution of problem \eqref{eq:CS-frlap} if $\tilde{u}$ has trace $f$ on $\Omega_e \times \{0\}$ and for all $\varphi \in \dot H^1_{0,\Omega_e}(\R^{n+1}_+,x_{n+1}^{1-2s})$ it holds
\begin{align*}
-\int_{\R^{n+1}_+} x_{n+1}^{1-2s} \nabla\varphi\cdot\nabla \tilde{u} \ d(x,x_{n+1}) = c_s^{-1} \int_{\R^n} \varphi(x,0) F(x,0) \ dx.
\end{align*}
Here,
\begin{align*}
\dot{H}^1_{0,\Omega_e}(\R^{n+1}_+,x_{n+1}^{1-2s}):= \left\{  \tilde{u} \in \dot{H}^1(\R^{n+1}_+, x_{n+1}^{1-2s}): \ \tilde{u}(x,0)= 0 \mbox{ for a.e. } x \in \Omega_{e} \right\}.
\end{align*}
Existence and uniqueness of $\tilde{u}$ follow by an application of Lax-Milgram's theorem. Additionally, it holds that
$$ \|\tilde u\|_{\dot H^1(\R^{n+1}_+, x_{n+1}^{1-2s})} \lesssim \|f\|_{H^s(\R^n)} + \|F\|_{H^{-s}(\Omega)}, $$
where we have made use of the compact support of $u-f$ on $\R^n$ and Poincar\'e's inequality.
By virtue of the trace estimates we obtain
$$ \|u\|_{\dot{H}^s(\R^n)}\lesssim \|\tilde u\|_{\dot H^1(\R^{n+1}_+, x_{n+1}^{1-2s})} \qquad \mbox{and} \qquad  \|\lim\limits_{x_{n+1} \rightarrow 0} x_{n+1}^{1-2s} \p_{n+1} \tilde{u}\|_{\dot{H}^{-s}(\R^n)}\lesssim \|\tilde u\|_{\dot H^1(\R^{n+1}_+, x_{n+1}^{1-2s})}.$$
Indeed, both estimates follow from a tangential Fourier transform of the Caffarelli-Silvestre extension, as shown in \cite{CS07}.

Moreover, \cite{ST10} provides a solution representation for $\tilde{u}$ by a Neumann kernel (Theorem 1.1 in \cite{ST10})
\begin{align*}
\tilde{u}(x,x_{n+1}) = \frac{1}{\Gamma(s)} \int_0^\infty \int_{\R^n} K_t(x,z) ((-\Delta)^s u)(z) dz \ e^{-\frac{x_{n+1}^2}{4t}} \frac{dt}{t^{1-s}},
\end{align*}
where $K_t(x,z)$ denotes the heat kernel. This can also be observed by a Fourier transform in the tangential directions or, alternatively, the explicit form of the heat kernel. If we consider $x_{n+1}=0$ we deduce
\begin{align*}
u(x) = \frac{1}{\Gamma(s)} \int_0^\infty \int_{\R^n} K_t(x,z) ((-\Delta)^s u)(z) dz \frac{dt}{t^{1-s}}.
\end{align*}

\subsection{The variable coefficient fractional Laplacian}\label{ssec:prel-VariableCoeffFrLapl}

In this section we define a variable coefficient, anisotropic fractional Laplacian. Following the structure of the previous section, we rely on a variable coefficient Caffarelli-Silvestre extension \cite{CS07, ST10}. 

We let $\tilde{a}= \begin{pmatrix} a & 0 \\ 0 & 1 \end{pmatrix}$, where $a \in L^{\infty}(\R^n;\R^{n \times n})$ is a uniformly elliptic, symmetric tensor field. By energy (or Lax-Milgram) arguments it follows that
\begin{equation*}
\begin{cases}
\begin{alignedat}{2}
\nabla \cdot x_{n+1}^{1-2s} \tilde{a} \nabla \tilde{u} & = 0 \quad &&\mbox{in } \R^{n+1}_+,\\
\tilde{u} & = u \quad &&\mbox{in } \R^n \times \{0\},
\end{alignedat}
\end{cases}
\end{equation*}
has a unique solution $\tilde u\in \dot H^1(\R^{n+1}_+, x_{n+1}^{1-2s})$ for all $u\in H^s(\R^n)$, and the trace estimates $$ \|u\|_{H^s(\R^n)}\lesssim \|\tilde u\|_{\dot H^1(\R^{n+1}_+, x_{n+1}^{1-2s})}, \qquad \|\lim\limits_{x_{n+1} \rightarrow 0} x_{n+1}^{1-2s} \p_{n+1} \tilde{u}\|_{H^{-s}(\R^n)}\lesssim \|\tilde u\|_{\dot H^1(\R^{n+1}_+, x_{n+1}^{1-2s})}$$
hold. The first estimate follows as in the previous section, the second one can be obtained by a weak formulation of the normal derivative.
With this in mind and in analogy to the case of the fractional Laplacian, we define the operator $(-\nabla \cdot a \nabla)^s: H^s(\R^n)\rightarrow H^{-s}(\R^n)$ (in weak sense) as
\begin{align*}
(-\nabla \cdot a \nabla)^s u := -c_s \lim\limits_{x_{n+1}\rightarrow 0} x_{n+1}^{1-2s}\p_{n+1} \tilde{u}.
\end{align*}
We will thus view the nonlocal problem
\begin{equation*}
\begin{cases}
\begin{alignedat}{2}
(-\nabla \cdot a \nabla)^s u & = F \quad &&\mbox{in } \Omega,\\
u & = f \quad &&\mbox{in } \Omega_e
\end{alignedat}
\end{cases}
\end{equation*}
with $f\in H^s(\R^n)$ and $F\in H^{-s}(\Omega)$ locally as
\begin{equation}\label{eq:CS-varfrlap}
\begin{cases}
\begin{alignedat}{2}
\nabla \cdot x_{n+1}^{1-2s} \tilde{a} \nabla \tilde{u} & = 0 \quad &&\mbox{in } \R^{n+1}_+,\\
-c_s \lim\limits_{x_{n+1} \rightarrow 0} x_{n+1}^{1-2s} \p_{n+1}\tilde{u} & = F \quad &&\mbox{in } \Omega,\\
\tilde{u} & = f \quad &&\mbox{in } \Omega_e \times \{0\}.
\end{alignedat}
\end{cases}
\end{equation}
Similarly as above $\tilde{u} \in \dot H^1(\R^{n+1}_+, x_{n+1}^{1-2s})$ is the unique solution of \eqref{eq:CS-varfrlap} if it has trace $f$ on $\Omega_e \times \{0\}$ and for all $\varphi \in \dot H^1_{0,\Omega_e}(\R^{n+1}_+, x_{n+1}^{1-2s})$ it holds
\begin{align*}
-\int_{\R^{n+1}_+} x_{n+1}^{1-2s} \nabla\varphi \cdot \tilde{a}\nabla \tilde{u} \ d(x,x_{n+1}) = c_s^{-1} \int_{\R^n} \varphi(x,0) F(x,0) \ dx.
\end{align*}
We again have by the well-posedness of the problem that
\begin{align*}
\|u\|_{ H^s(\R^{n})} + \|\tilde u\|_{\dot H^1(\R^{n+1}_+, x_{n+1}^{1-2s})}  \lesssim \|f\|_{H^s(\R^n)} + \|F\|_{H^{-s}(\Omega)}.
\end{align*}
We remark that also the variable coefficient problem allows for various equivalent definitions, including spectral ones \cite{ST10} as well as kernel based ones \cite{GLX17}.

\subsection{The fractional conductivity equation}\label{ssec:prel-FracCondEq}

In this subsection we turn to a preliminary study of the fractional divergence and gradient operators $(\nabla\cdot)^s, \nabla^s$, as well as to a brief discussion of the fractional conductivity problem 
\begin{align*}
\begin{cases}
\begin{alignedat}{2}
\left( (\nabla \cdot)^s (\Theta \cdot\nabla^s) + q \right) u &= 0 \qquad &&\text{in } \Omega,\\
u &= f \qquad &&\text{in }  \Omega_e,
\end{alignedat}
\end{cases}
\end{align*}
for $\Omega \subset \R^n$ open and Lipschitz.

It is worth noticing immediately that the operators $(\nabla \cdot)^s (\Theta \cdot\nabla^s)$ and $(-\nabla\cdot a\nabla)^s$ represent two different fractional generalizations of the classical conductivity operator whose associated inverse problems are not equivalent.

The fractional gradient and divergence operators were introduced as part of the nonlocal vector calculus studied in \cite{DGLZ12, DGLZ13}, and were later further developed in the papers \cite{C20, C20a, CMR21, CRZ22, CdHS22, C22, RZ24}, among others. We refer to these works as main references for this section. Assume $s\in (0,1)$. We define the fractional gradient of a function $u\in C^\infty_c(\mathbb R^n)$ as the following two-point function:
$$ \nabla^s u (x,y) := \frac{c_{n,s}^{1/2}}{\sqrt{2}} \frac{u(y)-u(x)}{|x-y|^{n/2+s+1}}(x-y).$$
Having observed that this implies $\|\nabla^s u\|_{L^2(\mathbb R^{2n})} \leq \|u\|_{H^s(\R^n)}$, we extend the above definition by density to act as $\nabla^s : H^s(\R^n)\rightarrow L^2(\R^{2n})$. Moreover, we define the fractional divergence $(\nabla\cdot)^s$ as the adjoint of $\nabla^s$, hence deducing the mapping property $(\nabla\cdot)^s: L^2(\mathbb R^{2n})\rightarrow H^{-s}(\mathbb R^n)$. With these definitions, it is immediate to observe that in $H^s(\mathbb R^n)$ it holds that $$(\nabla\cdot)^s\nabla^s = (-\Delta)^s.$$ 

Let $\gamma\in C^\infty(\R^n)$ be a scalar function satisfying
$\gamma^{1/2}-1 \in C^\infty_c(\Omega)$ and $\gamma\geq \underline\gamma$ for some $\underline\gamma \in (0,\infty)$. These properties imply that $\gamma$ is bounded, and also $\gamma\equiv 1$ in $\Omega_e$. We refer to the function $\gamma$ as the \emph{conductivity}, while we will call $\Theta(x,y):=\gamma(x)^{1/2}\gamma(y)^{1/2}I$ the \emph{interaction matrix}, according to the terminology of nonlocal vector calculus. One immediately sees that $\Theta\cdot\nabla^s u \in L^2(\R^{2n})$ for all $u\in H^s(\R^n)$, and thus we can define the \emph{fractional conductivity operator} $$(\nabla\cdot)^s(\Theta\cdot\nabla^s) : H^s(\R^n)\rightarrow H^{-s}(\R^n).$$
To this operator we associate the direct problem
\begin{align}\label{direct-problem-cond}
\begin{cases}
\begin{alignedat}{2}
\left( (\nabla \cdot)^s (\Theta \cdot\nabla^s) + q \right) u &= F \qquad &&\text{in } \Omega,\\
u &= f \qquad &&\text{in }  \Omega_e,
\end{alignedat}
\end{cases}
\end{align}
where $q\in L^{\frac{n}{2s}}(\Omega)$, $F\in H^{-s}(\Omega)$, and $f\in H^s(\R^n)$. The well-posedness of \eqref{direct-problem-cond} in $H^s(\R^n)$ has been obtained in the case $q\equiv 0$ by energy methods in \cite{C20a}, see also \cite{DGLZ13}.  In this paper, we will always assume that $q$ is such that $0$ is not a Dirichlet eigenvalue for \eqref{direct-problem-cond}, which is always the case if for example $q\geq 0$ or $\|q\|_{L^{\frac{n}{2s}}(\Omega)}$ is small enough (see condition (Aq) and Section \ref{subsec:condition-Aq}). This allows to extend the well-posedness result also for non-vanishing $q$ correspondingly.\\

Consider now the following bilinear form associated with \eqref{direct-problem-cond}:
$$ B_{\gamma, q}(u,v) := \int_{\R^n} \int_{\R^n} \left( \Theta\cdot\nabla^s u \right) \cdot \nabla^s v \, dx \, dy + \int_\Omega qu v \, dx.$$
We see that $B_{\gamma, q}$ is defined on $H^s(\R^n)\times H^s(\R^n)$ and symmetric. Further, let $u_f\in H^s(\R^n)$ indicate the unique solution of \eqref{direct-problem-cond} in the case $F\equiv 0$. It is easy to prove (see e.g. \cite{C20a}) that the linear map $\Lambda_{\gamma, q}: H^s(\Omega_e) \rightarrow \left(H^s(\Omega_e)\right)^*$ given by
$$\langle \Lambda_{\gamma, q}(f|_{\Omega_e}),g|_{\Omega_e} \rangle := B_{\gamma,q}(u_f,g), \qquad f,g\in H^s(\R^n),$$
where $\langle\cdot,\cdot\rangle$ denotes the $H^s(\Omega_e)$-$\left(H^s(\Omega_e)\right)^*$ duality pairing, is well-defined, continuous and self-adjoint. We call $\Lambda_{\gamma, q}$ the \emph{Dirichlet-to-Neumann map} associated with \eqref{direct-problem-cond}. In this framework, before turning to its instability properties, we first address the uniqueness question in the associated fractional Calder\'on problem.

\begin{proposition}\label{uniqueness-frac-cond}
    Let $\Omega, W_1, W_2\subset \R^n$ be bounded, open and mutually disjoint sets. Let $\gamma\in C^\infty(\R^n)$ be a fixed conductivity function satisfying
$\gamma^{1/2}-1 \in C^\infty_c(\Omega)$ and $\gamma\geq \underline\gamma$ for some $\underline\gamma \in (0,\infty)$. Assume that $q_1, q_2 \in L^{\frac{n}{2s}}(\Omega)$ are such that
$$ \langle\Lambda_{\gamma, q_1} (f_1|_{\Omega_e}),f_2|_{\Omega_e}\rangle = \langle\Lambda_{\gamma, q_2} (f_1|_{\Omega_e}),f_2|_{\Omega_e}\rangle$$
for all $f_j\in C^\infty_c(W_j)$, $j=1,2$. Then $q_1=q_2$ in $\Omega$.
\end{proposition}

The proof of Proposition \ref{uniqueness-frac-cond} is given in Section \ref{SectionProofCondEq} and with this uniqueness result in hand, we then turn to the instability result. We remark that the uniqueness question in Proposition \ref{uniqueness-frac-cond} differs somewhat from the one investigated in \cite{C20a, CRZ22} and subsequent works, as in our case the conductivity $\gamma$ is fixed a priori, and a potential $q\neq 0$ is to be recovered. In this sense the question is rather closer to the original one asked in \cite{GSU20}, of which it represents a possible generalization to operators of non-constant coefficients. We highlight that uniqueness for this inverse problem has been implicitly solved in \cite[Section 6]{C20a} in the special case $A\equiv 0$. However, the proof of that result is made quite technical by the introduction of the magnetic potential $A$ and the stricter regularity assumptions. Therefore, for the convenience of the reader, we have included the argument for Proposition \ref{uniqueness-frac-cond} in Section \ref{SectionProofCondEq} below.

\subsection{On the condition (Aq)}\label{subsec:condition-Aq}
Having discussed the constant coefficient, the variable coefficient and the conductivity version of the fractional Calder\'on problem, let us briefly comment on the assumption (Aq). This condition allows to consider well-defined Dirichlet-to-Neumann operators in all three instances. 

More precisely, it guarantees that for $f \in H^{s}(\Omega_e)$, $F\in (H^{s}(\Omega))^{\ast}$ and $\tilde{q}\in B_{r_0}^{L^{\frac{n}{2s}}(\Omega)}(q)$ with $q$ satisfying the condition (Aq), problems such as
\begin{equation}
\begin{cases}
\begin{alignedat}{2}
\nabla \cdot x_{n+1}^{1-2s} \tilde{a} \nabla \tilde{u} & = 0 \quad &&\mbox{in } \R^{n+1}_+,\\
-c_s \lim\limits_{x_{n+1} \rightarrow 0} x_{n+1}^{1-2s} \p_{n+1}\tilde{u} + \tilde{q} u & = F \quad &&\mbox{in } \Omega,\\
\tilde{u} & = f \quad &&\mbox{in } \Omega_e \times \{0\},
\end{alignedat}
\end{cases}
\end{equation}
can be uniquely solved. In particular, viewing these equations as giving rise to bounded operators 
\begin{align*}
T_{\tilde{q}} : H^{s}(\Omega_e) \times (H^{s}(\Omega_e))^{\ast} \mapsto H^{s}(\R^n), \ (f,F) \mapsto  u:= \tilde{u}|_{\R^n \times \{0\}}
\end{align*}
by the Fredholm alternative, the condition (Aq) ensures that for all $\tilde{q}\in B_{r_0}^{L^{\frac{n}{2s}}(\Omega)}(q)$ these operators are well-defined and bounded.

Special examples of settings in which the condition (Aq) is satisfied are given by
\begin{itemize}
\item $q\in L^{\infty}$ with $q\geq 0$,
\item $\|q\|_{L^{\infty}(\Omega)} \leq \lambda_1(\Omega)$, where $\lambda_1(\Omega)$ denotes the smallest eigenvalue of the operator with vanishing potential in $\Omega$.
\end{itemize}
We remark that due to the discrete spectrum of the fractional Laplacian on bounded domains in all its three formulations, generically the condition (Aq) is satisfied.
Without this condition, Cauchy data would provide the natural generalization.

\subsection{Additional notation}\label{ssec:prel-Notation}

We conclude this section by introducing some more notation. Let $A,B \geq 0$. When we write $A \lesssim B$, we mean that there exists some constant $C>0$ such that
\begin{align*}
A \leq CB.
\end{align*}
Similarly, we write $A \sim B$, when we mean that there exists a constant $C>0$ such that
\begin{align*}
\frac{1}{C}B \leq A \leq CB.
\end{align*}

\section{Instability for the constant coefficient fractional Calderón problem}
\label{sec:instab}

In this section we prove the instability result in the case of constant coefficients (Theorem \ref{ThmInstability}). We will see that analytic regularization for a suitable comparison operator leads to exponential instability.

\subsection{The general instability framework}

We first provide the general arguments leading to the instability result.\\
The main argument is based on the following abstract result from \cite{KRS21} which will allow us to deduce the desired instability result ``by comparison'' with a suitably defined auxiliary operator.

\begin{theorem}[Theorem 5.14 in \cite{KRS21}]\label{Thm5.14[KRS21]}
Let $X$ be a Banach space, $Y$ a separable Hilbert space and let $F: X \to B(Y,Y')$ be a continuous map. Let $X_1 \subset X$ be a closed subspace and $K = \{u \in X: \Vert u \Vert_{X_1} \leq r \}$ for some $r>0$. Assume that $i: X_1 \to X$ is compact with $e_k(i) \gtrsim k^{-m'}$ for some $m'>0$. Also, assume that there exists an orthonormal basis $(\varphi_j)_{j\in\N}$ of $Y$ and constants $C,\rho,\mu > 0$ so that $F(u)$ and $F(u)^*$ satisfy
\begin{align}\label{Thm5.14[KRS21]Eq1}
\Vert F(u) \varphi_k \Vert_{Y'}, \Vert F(u)^* \varphi_k \Vert_{Y'} \leq Ce^{-\rho k^\mu}, \qquad k \geq 1,
\end{align}
uniformly over $u \in K$. Then there exists $c>0$ with the property that for any $\varepsilon>0$ small enough there are $u_1$, $u_2$ such that
\begin{align*}
\Vert u_1 - u_2 \Vert_X \geq \varepsilon, \quad \Vert u_j \Vert_{X_1} \leq r, \quad \Vert F(u_1) - F(u_2) \Vert_{B(Y,Y')} \leq \exp(-c \varepsilon^{-\frac{\mu}{m'(\mu+2)}}).
\end{align*}
In particular, if one has the stability property
\begin{align*}
\Vert u_1 - u_2 \Vert_X \leq \omega(\Vert F(u_1) - F(u_2) \Vert_{B(Y,Y')} ), \qquad u_1, u_2 \in K,
\end{align*}
then necessarily $\omega(t) \gtrsim \vert \log t \vert^{-\frac{m'(\mu+2)}{\mu}}$ for $t$ small.
\end{theorem}

\begin{remark}\label{RemarkPotentialsNotCenteredAt0}
We remark that in the proof of this result in \citep{KRS21} the condition $\Vert u \Vert_{X_1} \leq r$ in the definition of $K$ is just used in order to have compactness and some uniform bounds over $K$. In particular, the set $K$ does not need to be centered at $0$, which will be used in the proof of Theorem \ref{ThmInstability}.
\end{remark}

Let $\Omega, \Omega', W \subset \R^n$ and $\bar{q}\in L^{\frac{n}{2s}}(\Omega)$ be as in the statement of Theorem \ref{ThmInstability}. We will apply Theorem \ref{Thm5.14[KRS21]} to $F = \Gamma$, with
\begin{align*}
\Gamma: L^{\frac{n}{2s}}(\Omega)\cap B_{r_0}^{L^{\frac{n}{2s}}(\Omega)}(\bar{q}) \rightarrow B(\widetilde{H}^s(W), H^{-s}(W)), \qquad q \mapsto \Gamma(q) := \Lambda_q - \Lambda_0,
\end{align*}
where $r_0>0$ and $\bar{q}$ satisfy the conditions in (Aq). We remark that $\Lambda_0$ is also well-defined since zero is not a Dirichlet eigenvalue of $(-\Delta)^s$ on $\Omega$.
In order to show that $\Gamma$ satisfies the bounds \eqref{Thm5.14[KRS21]Eq1}, we will introduce a linear comparison operator $A: \widetilde{H}^s (W) \to H^s(\Omega')$ which dominates the nonlinear (in $q$) operator $\Gamma$, in the sense that
\begin{align*}
\Vert \Gamma(q) f \Vert_{H^{-s}(W)}, \Vert \Gamma(q)^* f \Vert_{H^{-s}(W)} \leq \Vert Af \Vert_{H^s(\Omega')}
\mbox{ for all } f\in \widetilde{H}^s(W)
\end{align*}
holds uniformly over $q \in K$ for some subset $K \subset L^{\frac{n}{2s}}(\Omega)$. We then show that the singular values of this operator $A$ decay exponentially and choose $\varphi_k$ in Theorem \ref{Thm5.14[KRS21]} as the singular value basis of $A$. This then guarantees the validity of the condition \eqref{Thm5.14[KRS21]Eq1}.

After a suitable operator $A$ has been identified, it thus suffices to prove decay of the associated entropy estimates. To this end, we will rely on regularization properties of the operator $A$, whose derivation will constitute a main part of this section.

\begin{lemma}\label{LemmaSingValuesOfASmoothCase}
Let $W,\Omega' \subset \R^n$ be open, non-empty, Lipschitz domains and let $s_1,s_2 \in \R$. Let $A:\widetilde{H}^{s_2}(W) \to H^{s_1}(\Omega')$ be a compact linear operator with singular values $\sigma_k(A)$. Assume that $A(B_1^{\widetilde{H}^{s_2}(W)}) \subset A^{\sigma,\rho}(\overline{\Omega'})$ for some $\sigma\geq1$ and $\rho>0$ is bounded. Then there exists $c>0$ such that
\begin{align*}
\sigma_k(A) \lesssim \exp(-ck^{\frac{1}{n\sigma}}).
\end{align*}
\end{lemma}

\begin{proof}
We factorize the operator $A = j \circ \widetilde{A}$, where $\widetilde{A}: \widetilde{H}^{s_2}(W) \rightarrow A^{\sigma,\rho}(\overline{\Omega}')$ and where $j$ denotes the inclusion $j: A^{\sigma,\rho}(\overline{\Omega}') \rightarrow H^{s_1}(\overline{\Omega}')$. Then,
\begin{align*}
\sigma_k(A) \leq \sigma_k(j)\|\widetilde{A}\| \leq C \|\widetilde{A}\| e^{- c k^{\frac{1}{n\sigma}}} ,
\end{align*}
where we used that $e_k(j) \lesssim e^{- \tilde{c}k^{\frac{1}{n\sigma +1}}}$ (recall Lemma \ref{lem:entropy_numbers_embedding}) implies that $\sigma_k(j) \lesssim e^{- ck^{\frac{1}{n\sigma}}}$ (cf. Lemma \ref{Lemma3.9[KRS21]} or the original reference in \cite[Lemma 3.9]{KRS21}).
\end{proof}

Given Lemma \ref{LemmaSingValuesOfASmoothCase}, we observe that in order to prove Theorem \ref{ThmInstability} it suffices to deduce (analytically) regularizing properties for a suitable comparison operator $A$. As shown in the next results, it is possible to choose $A$ to be a bounded linear operator and to consider associated regularity estimates.

\begin{proposition}\label{prop:UniformBounds}
Let $s\in (0,1)$. Let $\Omega \subset \R^n$ be open, non-empty, bounded and Lipschitz, $\Omega'\subset\Omega$ be compact, non-empty, Lipschitz and $W\subset\Omega_e$ open, non-empty, Lipschitz. Let $A: \widetilde{H}^s(W) \to H^s(\Omega')$ be defined as
\begin{align*}
A(f) = \restr{u_0}{\Omega'},
\end{align*}
where $u_0$ is the solution to
\begin{equation*}\label{EqPropAnalyticityFractCalderon}
\begin{cases}
\begin{alignedat}{2}
 (-\Delta)^s  u_0 &= 0 \qquad &&\text{in } \Omega,\\
u_0 &= f \qquad &&\text{in } \Omega_e.
\end{alignedat}
\end{cases}
\end{equation*}
Then $A(\widetilde{H}^s(W))\subset G^1(\Omega')$, and there exist constants $C, R>0$ (depending only on $\Omega', \Omega, W,n$) such that
\begin{align}\label{EqPropUniformBounds}
\Vert D^\tau u \Vert_{L^\infty(\Omega')} \leq C \|f\|_{H^s(W)} R^{\vert\tau\vert} \vert\tau\vert^{\vert\tau\vert} \quad \mbox{for all } \tau\in\N_0^n
\end{align}
holds for all $u\in A(\widetilde{H}^s(W))$ with $u = A(f)$.
\end{proposition}

We postpone the proof of Proposition \ref{prop:UniformBounds} to Section \ref{sec:reg} and first show how this then implies our main result. 
To this end, we discuss a final auxiliary result which provides the desired exponential bounds for the singular values of the operator $A$.

\begin{proposition}\label{prop:singular_value}
Let $A$ be defined as in Proposition \ref{prop:UniformBounds} and let $\sigma_k$ be the singular values of $A$. Then, for some $c>0$ it holds that
\begin{align*}
\sigma_k(A) \lesssim \exp(-c k^{\frac{1}{n}}).
\end{align*}
\end{proposition}

\begin{proof}
Recalling Remark \ref{RemarkInclusionGevreySpace}, we find by Proposition \ref{prop:UniformBounds} that $A(\widetilde{H}^s(W)) \subset A^{1,\rho}(\overline{\Omega'})$ for some $\rho>0$ and $B_1^{\widetilde{H}^s(W)} \subset A^{1,\rho}(\Omega')$ is bounded. By the disjointness of the domains $\Omega'$ and $W$ combined with, for instance, the kernel representation of the fractional Laplacian, the operator $A$ is compact. An application of Lemma \ref{LemmaSingValuesOfASmoothCase} then implies the statement.
\end{proof}

Now we are finally in the position to prove Theorem \ref{ThmInstability}.

\begin{proof}[Proof of Theorem \ref{ThmInstability}]
We begin by noting that zero is not a Dirichlet eigenvalue of $(-\Delta)^s$ and that, by assumption, for the potential $\bar{q}$ the condition (Aq) with radius $r_0 > 0$ is satisfied. For $\Omega' \Subset \Omega$ Lipschitz, we then define 
\begin{align*}
K:= \{ q \in L^{\frac{n}{2s}} (\Omega) : q \in \bar{q} + B_{r_0}^{W_0^{\delta,\frac{n}{2s}}(\Omega')}(0) \}.
\end{align*}
In this case the operator $\Gamma(q) := \Lambda_q - \Lambda_0$ is well-defined for every $q \in K$. It is known that the embedding $i: W^{\delta,\frac{n}{2s}}(\Omega) \to L^{\frac{n}{2s}}(\Omega)$ is compact with $e_k(i) \gtrsim k^{-\delta/n}$, see, for instance, \cite[Theorem in Section 3.5]{EdmundsTriebel1996} (noting that for the Besov spaces $B^s_{p,q}(\R^n)$ we have $B^s_{p,p}(\R^n) = W^{s,p}(\R^n)$ (c.f. \cite[Proposition 14.40]{Leoni2009})).\\
Now, let $q \in K$, $f \in \widetilde{H}^s(W)$, and let $u \in H^s(\R^n)$ be a solution to
\begin{align}\label{Eq1ProofMainThm}
\begin{cases}
\begin{alignedat}{2}
\left( (-\Delta)^s + q \right) u  &= 0 \qquad &&\text{in } \Omega,\\
u &= f \qquad &&\text{in } \Omega_e,
\end{alignedat}
\end{cases}
\end{align}
and $u_0 \in H^s(\R^n)$ a solution to \eqref{Eq1ProofMainThm} with $q=0$. Then, $w := u - u_0$ solves
\begin{align*}
\begin{cases}
\begin{alignedat}{2}
\left( (-\Delta)^s + q \right) w &= -q u_0 \qquad &&\text{in } \Omega,\\
w &= 0 \qquad &&\text{in } \Omega_e,
\end{alignedat}
\end{cases}
\end{align*}
and we have $\Gamma(q)f = (\Lambda_q - \Lambda_0)f = \restr{(-\Delta)^s w}{W}$. Reformulating this by means of the Caffarelli-Silvestre extension for $w$ as in Section \ref{ssec:prel-frlap}, we find that $\tilde{w}$ solves
\begin{equation*}
\begin{cases}
\begin{alignedat}{2}
\nabla \cdot x_{n+1}^{1-2s} \nabla \tilde{w} &= 0 \quad &&\mbox{in } \R^{n+1}_+,\\
c_s \lim_{x_{n+1} \to 0} x_{n+1}^{1-2s} \partial_{n+1} \tilde{w} - qw &= q u_0 \quad &&\mbox{in } \Omega,\\
\tilde{w} &= 0 \quad && \mbox{in } \Omega_e \times \{0\},
\end{alignedat}
\end{cases}
\end{equation*}
with $\left( -\Delta \right)^s w := -c_s \lim_{x_{n+1} \to 0} x_{n+1}^{1-2s} \partial_{n+1} \tilde{w}$. With this in hand we have the following sequence of inequalities
\begin{equation}\label{EqEstimateForComparisonOperator}
\begin{aligned}
\Vert (-\Delta)^s w &\Vert_{H^{-s}(W)} \lesssim \Vert (-\Delta)^s w \Vert_{H^{-s}(\R^n)} \lesssim \Vert \lim_{x_{n+1} \to 0} x_{n+1}^{1-2s} \partial_{n+1} \tilde{w} \Vert_{H^{-s}(\R^n)} \lesssim \Vert \tilde{w} \Vert_{\dot{H}^1(\R^{n+1}_+, x_{n+1}^{1-2s})}\\
&\lesssim \Vert q u_0 \Vert_{H^{-s}(\Omega)} \lesssim \Vert q u_0 \Vert_{L^{\frac{2n}{n+2s}}(\Omega)} = \Vert q u_0 \Vert_{L^{\frac{2n}{n+2s}}(\Omega')} \lesssim \Vert q \Vert_{L^{\frac{n}{2s}}(\Omega')} \Vert u_0 \Vert_{L^{\frac{2n}{n-2s}}(\Omega')}\\
&\lesssim \Vert q \Vert_{L^{\frac{n}{2s}}(\Omega')} \Vert u_0 \Vert_{H^s(\Omega')} \lesssim \Vert u_0 \Vert_{H^s(\Omega')},
\end{aligned}
\end{equation}
where the third inequality follows from the trace estimates and the fourth inequality from the well-posedness estimate from Section \ref{ssec:prel-frlap} together with the condition (Aq). Since, by Sobolev inequality, the embedding of $H^s(\Omega')$ into $L^{\frac{2n}{n-2s}}(\Omega')$ is continuous, the space $L^{\frac{2n}{n+2s}}(\Omega') = \left( L^{\frac{2n}{n-2s}}(\Omega') \right)^*$ can be continuously embedded into $H^{-s}(\Omega')$. This gives the fifth inequality. For the sixth equality we recall that $q$ was supported on $\Omega'$ and inequality seven is a consequence of Hölder's inequality. The second to last inequality is again due to the Sobolev inequality.

We next consider the linear operator
\begin{align*}
A: \widetilde{H}^s(W) \to H^s(\Omega'), \qquad f \mapsto \restr{u_0}{\Omega'},
\end{align*}
from Proposition \ref{prop:UniformBounds}, and obtain
\begin{align}\label{ProofThmInstabilityEq1}
\Vert \Gamma(q) f \Vert_{H^{-s}(W)} = \Vert (-\Delta)^s w \Vert_{H^{-s}(W)} \leq C \Vert u_0 \Vert_{H^s(\Omega')} = C \Vert Af \Vert_{H^s(\Omega')}
\end{align}
with the constant $C$ being uniform over $K$.\\
Since by Proposition \ref{prop:singular_value} the operator $A$ satisfies for some $\bar{c}>0$
\begin{align}\label{ProofFractLaplSingularValueEstimate}
\sigma_k(A) \lesssim \exp(-\bar{c}k^{\frac{1}{n}}),
\end{align}
by choosing $f = \varphi_k$ in \eqref{ProofThmInstabilityEq1} for $(\varphi_k)_{k\in\N} \subset \widetilde{H}^s(W)$ the singular value basis for $A$, which forms an orthonormal basis of $\widetilde{H}^s(W)$, we can conclude that
\begin{align*}
\Vert \Gamma(q) \varphi_k \Vert_{H^{-s}(W)} \lesssim \Vert A \varphi_k \Vert_{H^s(\Omega')} \lesssim \vert \sigma_k(A) \vert \lesssim \exp(-\bar{c}k^{\frac{1}{n}}).
\end{align*}
Since $\Gamma(q)$ is self-adjoint, the same bound holds for the adjoint operators $\Gamma(q)^*$. Since $\Lambda_{q_1} - \Lambda_{q_2} = \Gamma(q_1) - \Gamma(q_2)$, Theorem \ref{Thm5.14[KRS21]} applied for $F=\Gamma$, $K$ as above (recall also Remark \ref{RemarkPotentialsNotCenteredAt0}), $m'=\frac{\delta}{n}$, $\rho=\bar{c}$ and $\mu=\frac{1}{n}$ then yields the existence of some $c>0$ such that for any $0 < \varepsilon < r_0$ there are potentials $q_1, q_2 \in K$ with 
\begin{align*}
\Vert \Lambda_{q_1} - \Lambda_{q_2} \Vert_{\widetilde{H}^s(W) \to H^{-s}(W)} &\leq \exp\left(-c \varepsilon^{-\frac{1}{\delta(2+\frac{1}{n})}}\right),\\
\Vert q_1 - q_2 \Vert_{L^{\frac{n}{2s}}(\Omega)} &\geq \varepsilon.
\end{align*}
The second part of the statement then also follows from the general instability framework of Theorem \ref{Thm5.14[KRS21]}.
\end{proof}

\subsection{Regularity results for the Caffarelli-Silvestre extension}
\label{sec:reg}

In this subsection we deduce analytic regularization properties for the operator $(-\Delta)^s$ in tangential directions in the interior of $\Omega$, i.e., we provide the proof of Proposition \ref{prop:UniformBounds}. For alternative elliptic arguments leading to interior analyticity we refer to \cite{FMMS22, FMMS23} in polygons and polyhedra and to \cite{R16} for general regularity results for the fractional Laplacian and related operators.

\begin{proof}[Proof of Proposition \ref{prop:UniformBounds}]
We use that by \cite{ST10}, the function $u$ can be expressed in terms of the heat kernel $K_t(x,z)$ in the form
\begin{align}\label{eq:int_heat}
u(x)= \tilde{c}_s \int\limits_{0}^{\infty} \int\limits_{\R^n} K_t(x,z) g(z) dz \frac{dt}{t^{1-s}},
\end{align}
where $g(z):= (-\Delta)^s u \in \dot{H}^{-s}(\R^n)$ for $f\in \widetilde{H}^{s}(W)$. Indeed, this follows from noting that the map 
\begin{align*}
T: \dot{H}^{-s}(\R^n) \rightarrow \dot{H}^{s}(\R^n), \ g \mapsto \tilde{c}_s \int\limits_{0}^{\infty} \int\limits_{\R^n} K_t(x,z) g(z) dz \frac{dt}{t^{1-s}}
\end{align*}
is well-defined and bounded (first for $g\in C_c^{\infty}(\R^n)$ and then in general by density). Hence, for $g = (-\D)^s u$ we also obtain
\begin{align*}
\|T g\|_{\dot{H}^s(\R^n)} \leq C \|u\|_{H^s(\R^n)},
\end{align*}
and, for $u \in C_c^{\infty}(\R^n)$ and such a choice of $g$, a direct computation shows that $u = Tg$. By density this then also holds for $u \in H^s(\R^n)$, which provides a proof of the identity from \cite{ST10} in our setting.

We claim that $u(x)$ is analytic for $x\in \Omega'$. In order to observe this, we split the definition in \eqref{eq:int_heat} into two parts and set $u(x) = v_1(x)+ v_2(x)$ with
\begin{align}\label{eq:int_heat1}
v_1(x)= \int\limits_{1}^{\infty} \int\limits_{\R^n} K_t(x,z) g(z) dz \frac{dt}{t^{1-s}},
\end{align}
\begin{align}\label{eq:int_heat2}
v_2(x)= \int\limits_{0}^{1} \int\limits_{\R^n} K_t(x,z) g(z) dz \frac{dt}{t^{1-s}}.
\end{align}
Furthermore we define $w(x,t):= \int\limits_{\R^n} K_t(x,z) g(z) dz$ for $t\geq 0$.\\

\emph{Step 1: Bounds for $\partial_x^\alpha v_1$, $\vert\alpha\vert \geq 1$.}
For $v_1$, we note that the following analytic estimates hold (cf. inequality (1.7) in \cite{EMZ17} and references therein): There exist $C,R>0$ such that for all $x\in \Omega'$, $t\geq 1$
\begin{align*}
|\p_x^{\alpha} w(x,t)| \leq C R^{|\alpha|}|\alpha|! \left(t-\frac{1}{2}\right)^{-\frac{\vert\alpha\vert}{2}- \frac{n}{4}} \|w\|_{L^2(\R^n \times \{1/2\})}.
\end{align*}
As a consequence, exchanging integration and differentiation for which we invoke the dominated convergence theorem, we obtain for $\alpha\in\N^n$ with $\vert\alpha\vert \geq 1$
\begin{align}\label{eq:analy1}
\left\vert \partial_x^\alpha v_1 (x_0) \right\vert \leq \int\limits_{1}^{\infty} \left| \p_x^{\alpha} w(x_0,t) \right| \frac{dt}{t^{1-s}} \leq C_s R^{|\alpha|}|\alpha|! \|w\|_{L^2(\R^n \times \{1/2\})},
\end{align}
where we have used that $n\geq2$ and thus $1-s+\frac{\vert\alpha\vert}{2}+\frac{n}{4}>1$ to get the integrability in $t$. Hence, we infer the desired analyticity bounds for $\partial_x^\alpha v_1$ (in the case of $\vert\alpha\vert\geq1$). It remains to bound the quantity on the right hand side of \eqref{eq:analy1} in terms of $f$. To this end, we note that
\begin{align*}
w(x,1/2) 
&= \int\limits_{\R^n} K_{\frac{1}{2}}(x,z) g(z) dz 
=  \F^{-1}( e^{-\frac{|\xi|^2}{2}} \hat{g})(x) ,
\end{align*}
and thus, by the Parseval identity,
\begin{align*}
\|w\|_{L^2(\R^n \times \{1/2\})} \leq \|e^{-\frac{|\xi|^2}{2}} \hat{g}\|_{L^2(\R^n)}
\leq \|e^{-\frac{|\xi|^2}{2}}\|_{\dot{H}^s(\R^n)}\| \hat{g}\|_{\dot{H}^{-s}(\R^n)}
\leq C \| f\|_{H^{s}(W)}.
\end{align*}

\emph{Step 2: Bounds for $\partial_x^\alpha v_2$.}
We next turn to the analyticity of $v_2$. This does not immediately follow as above, since the heat kernel is not analytically regularizing in time and the radius of convergence decreases exponentially as $t\rightarrow 0$. However, we use that $g = 0$ in $\Omega$. As a consequence, in $\Omega$ the function $w(x,t)$ solves the equation
\begin{align*}
\begin{cases}
\begin{alignedat}{2}
(\p_t - \Delta) w &= 0 \quad &&\mbox{in } \Omega \times (0,\infty),\\
w & = 0 \quad &&\mbox{on } \Omega \times \{0\}.
\end{alignedat}
\end{cases}
\end{align*}
Hence, for $x_0 \in \Omega'$ the function $w$ satisfies the estimate
\begin{align*}
|\p_x^{\alpha} w(x_0,t)| \leq C \|w\|_{L^2(\Omega'' \times [0,1])} R^{|\alpha|} |\alpha|! \quad \mbox{ for } x_0 \in \Omega', t\in [0,1).
\end{align*}
for some $C = C(n, \Omega', \Omega'')$, some $R = R(n,\Omega',\Omega'')\in (0,1)$ and with $\Omega'' \subset \R^n$ open, such that $\Omega' \Subset \Omega'' \Subset \Omega$. This follows analogously as in \cite[Lemma 2.2]{EMZ17} (without weights, which for convenience of the reader, we formulate and prove again below in Lemma \ref{LemmaLInftyEstimateHeatEq}).

It remains to estimate $\|w\|_{L^2(\Omega'' \times [0,1])}$ in terms of the data $\|f\|_{H^s(W)}$. To this end, we let $\psi \in C^{\infty}(\R^n)$ with $\supp(\psi) \in B_{\delta/2}(0)^c$ and $\psi \equiv 1$ in $B_\delta(0)^c$ with $\delta \in (0,1)$ satisfying $\delta< \dist(\supp(g),\Omega'')$. Then, using that $g = 0$ in $\Omega$, we note that
\begin{equation}
\label{eq:aux_est}
\begin{aligned}
\|w\|_{L^2(\Omega'' \times [0,1])}
&= \| \int\limits_{\R^n} K_t(\cdot-z) g(z) dz \|_{L^2(\Omega'' \times (0,1))}
=  \| \int\limits_{\R^n} K_t(\cdot-z) \psi(\cdot-z) g(z) dz \|_{L^2(\R^n \times (0,1))}\\
&\leq \|g\|_{\dot{H}^{-s}(\R^n)} \|K_t(\cdot) \psi(\cdot) \|_{\dot{H}^s(\R^n) \times L^2 (0,1)}
 \leq C \|f\|_{H^s(W)} \|K_t(\cdot) \psi(\cdot) \|_{\dot{H}^s(\R^n) \times L^2 (0,1)}\\
& \leq C \|f\|_{H^s(W)}.
\end{aligned}
\end{equation}
Here we have used the fact that in $\supp(\psi)$ and for $t\in (0,1)$ there exists $c>0$ such that
\begin{align*}
K_t(x), |\nabla K_t(x)| \leq c t^{-n/2-1} e^{- c/t}, 
\end{align*}
which is non-singular as $t \rightarrow 0$. Also observe that, as above, the constant $C>1$ in \eqref{eq:aux_est} may change from line to line.

We insert this bound into the expression for $v_2$, and infer the bounds for $\alpha \in \N^n$
\begin{align*}
\left\vert \partial_x^\alpha v_2(x_0,t) \right\vert &\leq \int\limits_{0}^{1} \left| \p_x^{\alpha} w(x_0,t) \right| \frac{dt}{t^{1-s}} \leq C \Vert w \Vert_{L^2(\Omega'' \times [0,1])} R^{|\alpha|} |\alpha|! \left| \int\limits_{0}^{1} \frac{dt}{t^{1-s}} \right|\\
&\leq C \Vert f \Vert_{H^s(W)} R^{|\alpha|} |\alpha|!.
\end{align*}

\emph{Step 3: Bounds for $u$.}
Combining the above estimates for $v_1$ and $v_2$ implies that $u$ satisfies for $\vert\alpha\vert \geq 1$
\begin{align}\label{Eq1AnalyticityStep3}
\Vert \partial_x^\alpha u \Vert_{L^\infty(\Omega')} \leq \widetilde{C} \Vert f \Vert_{H^s(W)} \widetilde{R}^{\vert\alpha\vert} \vert\alpha\vert^{\vert\alpha\vert}
\end{align}
when $f \in \widetilde{H}^{s}(W)$. By the well-posedness of \eqref{EqPropAnalyticityFractCalderon} we additionally know that $\Vert u \Vert_{L^2(\Omega')} \leq \Vert u \Vert_{H^s(\R^n)} \leq C \Vert f \Vert_{H^s(W)}$. We then deduce by Sobolev embedding that
\begin{equation}\label{Eq2AnalyticityStep3}
\begin{aligned}
\Vert u \Vert_{L^\infty(\Omega')} &\leq C \Vert u \Vert_{H^{n}(\Omega')} \leq \Vert u \Vert_{L^2(\Omega')} + \operatorname{Vol}(\Omega')^{1/2} \sum_{j=1}^n \Vert D_x^j u \Vert_{L^\infty(\Omega')} \leq C n \widetilde{C} \Vert f \Vert_{H^s(W)} \widetilde{R}^n n^n\\
&\leq C \Vert f \Vert_{H^s(W)}.
\end{aligned}
\end{equation}
Thus, combining \eqref{Eq1AnalyticityStep3} and \eqref{Eq2AnalyticityStep3}, we have
\begin{align*}
\Vert D^\tau u \Vert_{L^\infty(\Omega')} \leq C \|f \|_{H^s(W)} R^{\vert\tau\vert} \vert\tau\vert^{\vert\tau\vert} \quad \mbox{for all } \tau\in\N_0^n.
\end{align*}
\end{proof}

\section{Instability for the variable coefficient fractional Calderón problem}\label{sec:low_reg}

In this section, we present the proof of Theorem \ref{thm:low_reg_instab}. To this end, we recall some notation. For $\Omega \subset \R^n$ open, $R>0$ we denote $Q_{\Omega,R} := \Omega \times (0,R) \subset \R^{n+1}_+$ and $\Gamma_{\Omega,R} := \Omega \times [0,R) \subset \overline{\R^{n+1}_+}$. The spaces $H_0^1(\Gamma_{\Omega,R}, x_{n+1}^{1-2s})$ and $\dot{H}_0^1(\Gamma_{\Omega,R}, x_{n+1}^{1-2s})$ were defined by
\begin{align*}
H_0^1(\Gamma_{\Omega,R},x_{n+1}^{1-2s}) := \{ \tilde{u} \in H^1(Q_{\Omega,R}, x_{n+1}^{1-2s}): \tilde{u} = 0 \text{ on } \left( \partial\Omega \times(0,R) \right) \cup \left( \Omega \times \{R\} \right) \},
\end{align*}
and
\begin{align*}
\dot{H}_0^1(\Gamma_{\Omega,R},x_{n+1}^{1-2s}) := \{ \tilde{u} \in \dot{H}^1(Q_{\Omega,R}, x_{n+1}^{1-2s}): \tilde{u} = 0 \text{ on } \left( \partial\Omega \times(0,R) \right) \cup \left( \Omega \times \{R\} \right) \}.
\end{align*}

\subsection{Singular value and entropy estimates for embeddings adapted to the (variable coefficient) Caffarelli-Silvestre extension}
\label{sec:singular}

Before turning to the proof of Theorem \ref{thm:low_reg_instab}, in this section, we deduce singular value bounds adapted to embeddings associated with the variable coefficient fractional Caffarelli-Silvestre extension operator. This relies on a characterization of the relevant spaces in terms of sequence spaces based on eigenfunction decompositions of and Weyl type bounds for the Caffarelli-Silvestre extension.

Hence, we first recall the eigenvalue characterization for a variant of a variable coefficient Caffarelli-Silvestre operator.
Indeed, the following proposition follows exactly like Proposition 2.1 in \cite{FelliFerrero2020} with the obvious modifications.

\begin{proposition}\label{PropEigenfunctionForCSExtension}
Let $s\in(0,1)$, $R>0$ and let $\Omega \subset \R^n$ be open, non-empty, bounded and Lipschitz. Define 
\begin{align*}
\tilde{e}_{l,m}(x,x_{n+1}) := \gamma_m x_{n+1}^s J_{-s} \left( \frac{j_{-s,m}}{R}x_{n+1} \right) e_l(x) \quad \text{for any } l,m \in \N_{>0}
\end{align*}
and
\begin{align*}
\lambda_{l,m} := \mu_l + \frac{j_{-s,m}^2}{R^2} \quad \text{for any } l,m \in \N_{>0},
\end{align*}
where $J_{-s}$ is the Bessel function of first kind with order $-s$, $0<j_{-s,1}<j_{-s,2}<\dots<j_{-s,m}<\cdots$ are the zeros of $J_{-s}$, $\gamma_m := \left(\int_0^R x_{n+1} \left[J_{-s}(\frac{j_{-s,m}}{R}x_{n+1})\right]^2 dx_{n+1} \right)^{-1/2}$ is a normalizing constant, $\{e_l\}_{l\geq1}$ denotes a complete orthonormal system of eigenfunctions of $\left( -\Delta \right)$ in $\Omega$ with homogeneous Dirichlet boundary data and $\mu_1<\mu_2\leq\dots\leq\mu_l\leq\cdots$ the corresponding eigenvalues.\\
Then for any $l,m \in \N_{>0}$, $\tilde e_{l,m}$ is an eigenfunction of
\begin{equation}\label{EqDefiningEigenfunction}
\begin{cases}
\begin{alignedat}{2}
-\nabla \cdot x_{n+1}^{1-2s} \nabla \tilde{e}_{l,m} &= \lambda_{l,m} \tilde{e}_{l,m} \quad &&\text{in } Q_{\Omega,R},\\
\tilde{e}_{l,m} &= 0 \quad &&\text{on } \left( \partial\Omega \times (0,R) \right) \cup \left( \Omega \times \{R\} \right),\\
\lim_{x_{n+1}\to0} x_{n+1}^{1-2s} \partial_{n+1} \tilde{e}_{l,m} &= 0 \quad &&\text{on } \Omega,
\end{alignedat}
\end{cases}
\end{equation}
with corresponding eigenvalues $\lambda_{l,m}$ (interpreted in a weak sense). Moreover, the set $\{\tilde{e}_{l,m}: l,m \in \N_{>0}\}$ is a complete orthonormal system for $L^2(Q_{\Omega,R},x_{n+1}^{1-2s})$.
\end{proposition}

Building on the characterization of the spectrum of the Caffarelli-Silvestre operator, we derive a similar result as Proposition 2.7 in \cite{KRS21} which allows us to identify the relevant weighted Sobolev spaces with corresponding weighted sequence spaces. 

To this end, we relabel the eigenvalues and eigenfunctions from above and define these as  $(\lambda_k, \tilde{\varphi}_k) \subset \left( \R_{> 0} \times L^2(Q_{\Omega,R},x_{n+1}^{1-2s}) \right)$ to denote the ordered pairs of eigenvalues and eigenfunctions $(\lambda_{l,m}, \tilde{e}_{l,m})$ from Proposition \ref{PropEigenfunctionForCSExtension} such that $0 < \lambda_1 \leq \lambda_2 \leq \dots \leq \lambda_k \leq \cdots$. 
We note that $\tilde{\varphi}_k \in H_0^1(\Gamma_{\Omega,R},x_{n+1}^{1-2s})$ and that by their defining equation the functions $\tilde{\varphi}_k$ are also $H_0^1(\Gamma_{\Omega,R},x_{n+1}^{1-2s})$ orthogonal. Hence, defining the functions $\tilde{\varphi}_k \in H_0^1(\Gamma_{\Omega,R},x_{n+1}^{1-2s})$ as $L^2(Q_{\Omega,R},x_{n+1}^{1-2s})$ normalized, ordered eigenfunctions, we obtain an orthonormal basis $(\tilde{\varphi}_k)_{k \in \N}$ of $L^2(Q_{\Omega,R},x_{n+1}^{1-2s})$ which is also orthogonal with respect to the $\dot{H}_0^1(\Gamma_{\Omega,R},x_{n+1}^{1-2s})$ inner product. In what follows below, we will work with this basis.

\begin{proposition}\label{PropNormEquivalenceCSExtension}
Let $\Omega$ be open, non-empty, bounded and Lipschitz, and let $R>0$. Let $(\lambda_k, \tilde{\varphi}_k) \subset \left( \R_{> 0} \times H^1_0(\Gamma_{\Omega,R},x_{n+1}^{1-2s}) \right)$ be as above. Then for any $\tilde{u} \in H_0^1(\Gamma_{\Omega,R},x_{n+1}^{1-2s})$
\begin{align*}
\Vert \tilde{u} \Vert_{H^1(\Gamma_{\Omega,R},x_{n+1}^{1-2s})}^2 \sim \sum_{k=1}^\infty k^{\frac{2}{n+1}} \vert (\tilde{u},\tilde{\varphi}_k)_{L^2(Q_{\Omega,R},x_{n+1}^{1-2s})} \vert^2.
\end{align*}
Moreover, 
\begin{align*}
H_0^1(\Gamma_{\Omega,R},x_{n+1}^{1-2s}) = \left\{ \tilde{u} \in L^2(Q_{\Omega,R},x_{n+1}^{1-2s}): \ \sum_{k=1}^\infty k^{\frac{2}{n+1}} \vert (\tilde{u},\tilde{\varphi}_k)_{L^2(Q_{\Omega,R},x_{n+1}^{1-2s})} \vert^2 < \infty \right\}.
\end{align*}
In particular, there exists an isomorphism between $H_0^1(\Gamma_{\Omega,R},x_{n+1}^{1-2s})$ and the sequence space
\begin{align*}
h^1:= \left\{ (a_j)_{j \in \N} \in \ell^2: \ \sum\limits_{j=1}^{\infty} j^{\frac{2}{n+1}} |a_j|^2 < \infty \right\},
\end{align*}
given by $ \tilde{u} \mapsto ( (\tilde{u},\tilde{\varphi}_k)_{L^2(Q_{\Omega,R},x_{n+1}^{1-2s})})_{k \in \N}$.
\end{proposition}

\begin{proof}
\emph{Step 1: Weyl type asymptotics of the eigenvalues.}
We begin by deducing the asymptotic behaviour of the eigenvalues $\lambda_k$. We recall that $\lambda_{l,m}$ was given by
\begin{align*}
\lambda_{l,m} := \mu_l + \frac{j_{-s,m}^2}{R^2} \quad \text{for any } l,m \in \N_{>0}.
\end{align*}
By Weyl asymptotics (cf. \cite[Chapter 8 Corollary 3.5]{Taylor2010}), we know that $\mu_l \sim l^{2/n}$ and by McMahon's asymptotic expansion (cf. 10.21 (vi) in \cite{Olver10}) we also know that $j_{-s,m}^2 \sim m^2$, which implies that
\begin{align}\label{EqAsympLambdaLM}
\lambda_{l,m} \sim l^{2/n} + m^2.
\end{align}
Fix some $N \in \N_{>0}$ sufficiently large. We seek to estimate the number of pairs $(l,m)$ such that $l^{2/n}+m^2 \leq N$. For fixed $m \leq \sqrt{N}$ there are $\lfloor (N-m^2)^{n/2}\rfloor$ elements $l$ such that the condition is satisfied. Summing over $m$ we find
\begin{align*}
\sum_{m=1}^{\lfloor \sqrt{N} \rfloor} \lfloor (N-m^2)^{n/2} \rfloor \leq \sum_{m=1}^{\lfloor \sqrt{N} \rfloor} N^{n/2} \leq N^{\frac{n+1}{2}},
\end{align*}
and
\begin{align*}
\sum_{m=1}^{\lfloor \sqrt{N} \rfloor}\lfloor (N-m^2)^{n/2}\rfloor &\geq \sum_{m=1}^{\lfloor \sqrt{N}/2 \rfloor} \lfloor(N-m^2)^{n/2}\rfloor \geq \sum_{m=1}^{\lfloor \sqrt{N}/2 \rfloor} (N-\frac{1}{4}N)^{n/2} = \sum_{m=1}^{\lfloor \sqrt{N}/2 \rfloor} (\frac{3}{4}N)^{n/2}\\
&\geq \frac{1}{3} \left(\frac{3}{4}\right)^{n/2} N^{\frac{n+1}{2}}.
\end{align*}
Hence, there are approximately $N^{\frac{n+1}{2}}$ pairs $(l,m)$ such that $l^{2/n}+m^2 \leq N$, which implies together with \eqref{EqAsympLambdaLM} that
\begin{align*}
\lambda_k \sim k^{\frac{2}{n+1}}.
\end{align*}

\emph{Step 2: Characterization as a sequence space.}
We claim that as a consequence of Proposition \ref{PropEigenfunctionForCSExtension} we obtain for $\tilde{u} \in H_0^1(\Gamma_{\Omega,R},x_{n+1}^{1-2s})$ the norm equivalence
\begin{align}\label{EqNormEquivalence}
\Vert \tilde{u} \Vert_{H^1(\Gamma_{\Omega,R},x_{n+1}^{1-2s})}^2 \sim \sum_{k=1}^\infty k^{\frac{2}{n+1}} \vert (\tilde{u},\tilde{\varphi}_k)_{L^2(Q_{\Omega,R},x_{n+1}^{1-2s})} \vert^2.
\end{align}
Indeed, firstly, by the zero boundary conditions, we have that $H_0^1(\Gamma_{\Omega,R},x_{n+1}^{1-2s}) = \dot{H}_0^1(\Gamma_{\Omega,R},x_{n+1}^{1-2s})$. 
Next, if $\tilde{u} \in H_0^1(\Gamma_{\Omega,R},x_{n+1}^{1-2s})$, we obtain that 
\begin{align*}
\lambda_k (\tilde{u},\tilde{\varphi}_k)_{L^2(Q_{\Omega,R},x_{n+1}^{1-2s})} = (\nabla \tilde{u},\nabla \tilde{\varphi}_k)_{L^2(Q_{\Omega,R},x_{n+1}^{1-2s})}.
\end{align*}
Due to the $L^2(Q_{\Omega,R},x_{n+1}^{1-2s})$-orthonormality of the functions $\tilde{\varphi}_k$, the norm equivalence follows. Moreover, together with the singular value asymptotics from step 1, this also shows that any element of $H_0^1(\Gamma_{\Omega,R},x_{n+1}^{1-2s})$ can be identified with an element in 
\begin{align*}
X:=\left\{ \tilde{u} \in L^2(Q_{\Omega,R},x_{n+1}^{1-2s}): \ \sum_{k=1}^\infty k^{\frac{2}{n+1}} \vert (\tilde{u},\tilde{\varphi}_k)_{L^2(Q_{\Omega,R},x_{n+1}^{1-2s})} \vert^2 < \infty \right\}.
\end{align*}
It remains to prove that if $\tilde{u} \in X$, then also $\tilde{u}$ in $H_0^1(\Gamma_{\Omega,R},x_{n+1}^{1-2s}) $. This follows from the fact that since $\tilde{\varphi}_k \in H_0^1(\Gamma_{\Omega,R},x_{n+1}^{1-2s})$ for any $N \in \N$ also
\begin{align*}
\tilde{u}_N := \sum_{k=1}^N k^{\frac{2}{n+1}}  (\tilde{u},\tilde{\varphi}_k)_{L^2(Q_{\Omega,R},x_{n+1}^{1-2s})}  \tilde{\varphi}_k \in H_0^1(\Gamma_{\Omega,R},x_{n+1}^{1-2s}),
\end{align*}
and, by the norm equivalence from above and since $\tilde{u} \in X$, for $N\geq M$
\begin{align*}
\Vert \tilde{u}_N - \tilde{u}_M \Vert_{H^1(\Gamma_{\Omega,R},x_{n+1}^{1-2s})}^2 \sim \sum_{k=M}^N k^{\frac{2}{n+1}} \vert (\tilde{u},\tilde{\varphi}_k)_{L^2(Q_{\Omega,R},x_{n+1}^{1-2s})} \vert^2 \rightarrow 0 \mbox{ as } M,N \rightarrow \infty.
\end{align*}
Therefore, $(\tilde{u}_N)_{N\in \N}$ forms a Cauchy sequence in $H_0^1(\Gamma_{\Omega,R},x_{n+1}^{1-2s})$ and thus by the completeness of $H_0^1(\Gamma_{\Omega,R},x_{n+1}^{1-2s})$ the limit $N \rightarrow \infty$ exists and is an element of $H_0^1(\Gamma_{\Omega,R},x_{n+1}^{1-2s})$.

Last but not least, the identification with the sequence space $h^{1}$ follows directly from the characterization as the function space $X$ from above.
\end{proof}

Using the above characterization in terms of sequence spaces, we obtain two-sided singular value estimates for the embedding $i: H_0^1(\Gamma_{\Omega,R}, x_{n+1}^{1-2s}) \to L^2(Q_{\Omega,R}, x_{n+1}^{1-2s})$. 

\begin{proposition}\label{PropEstimateSingularValuesZeroBoundaryData}
Let $\Omega$ be open, non-empty, bounded and Lipschitz, and let $R>0$. The embedding $i: H_0^1(\Gamma_{\Omega,R}, x_{n+1}^{1-2s}) \to L^2(Q_{\Omega,R}, x_{n+1}^{1-2s})$ satisfies the singular value bounds
\begin{align*}
\sigma_k(i) \sim k^{-\frac{1}{n+1}}.
\end{align*}
\end{proposition}

\begin{proof}
Let $\tilde{u} \in H_0^1(\Gamma_{\Omega,R},x_{n+1}^{1-2s})$, $(\tilde{\varphi}_j)_{j\in\N_{>0}}$ the orthonormal system as in Proposition \ref{PropNormEquivalenceCSExtension} and $b\in\Z$. Define $J^b$ by
\begin{align*}
J^b\tilde{u} := \sum_{j=1}^\infty j^{\frac{b}{n+1}} ( \tilde{u},\tilde{\varphi}_j )_{L^2(Q_{\Omega,R},x_{n+1}^{1-2s})} \tilde{\varphi}_j.
\end{align*}
Then, by Proposition \ref{PropNormEquivalenceCSExtension}, we have $\Vert \tilde{u} \Vert_{H^1(\Gamma_{\Omega,R},x_{n+1}^{1-2s})} \sim \Vert J^1 \tilde{u} \Vert_{L^2(Q_{\Omega,R},x_{n+1}^{1-2s})}$ and $H_0^1(\Gamma_{\Omega,R},x_{n+1}^{1-2s})$ and $L^2(\Gamma_{\Omega,R},x_{n+1}^{1-2s})$ are isomorphic to the sequence spaces $h^1$ and $\ell^2$ with the isomorphism given by $u \mapsto ((u, \tilde{\varphi}_j)_{L^2(Q_{\Omega,R}, x_{n+1}^{1-2s})})_{j \in \N}$. We note that for all $\tilde{u}, \tilde{v} \in H_0^1(\Gamma_{\Omega,R},x_{n+1}^{1-2s})$ it holds
\begin{align*}
( i^* i(\tilde{u}), \tilde{v} )_{H^1(\Gamma_{\Omega,R},x_{n+1}^{1-2s})} = ( i(\tilde{u}), i(\tilde{v}) )_{L^2(Q_{\Omega,R},x_{n+1}^{1-2s})} = ( J^{-2} \tilde{u}, \tilde{v} )_{H^1(\Gamma_{\Omega,R},x_{n+1}^{1-2s})}.
\end{align*}
Thus, $i^*i = J^{-2}$ on $H_0^1(\Gamma_{\Omega,R},x_{n+1}^{1-2s})$, and therefore we have
\begin{align*}
\sigma_k(i)^2 = \sigma_k(i^* i) = \sigma_k(J^1 \circ i^* i \circ J^{-1}) = \sigma_k \left( J^{-2}: L^2(Q_{\Omega,R},x_{n+1}^{1-2s}) \to L^2(Q_{\Omega,R},x_{n+1}^{1-2s}) \right) = k^{-\frac{2}{n+1}},
\end{align*}
where we used in the last step that $J^{-2}: L^2(Q_{\Omega,R},x_{n+1}^{1-2s}) \to L^2(Q_{\Omega,R},x_{n+1}^{1-2s})$ is a diagonal operator.
\end{proof}

Last but not least, by an extension argument the above discussion also gives rise to singular value estimates for the  embedding $i: H^1(Q_{\Omega,R}, x_{n+1}^{1-2s}) \to L^2(Q_{\Omega,R}, x_{n+1}^{1-2s})$. It will be these bounds, which we use in the iterative compression gain in the following section.

\begin{proposition}\label{PropEstimateSingularValues}
Let $\Omega$ be open, non-empty, bounded and Lipschitz, and let $R>0$. The embedding $i: H^1(Q_{\Omega,R}, x_{n+1}^{1-2s}) \to L^2(Q_{\Omega,R}, x_{n+1}^{1-2s})$ satisfies the entropy bounds
\begin{align*}
\sigma_k(i) \lesssim k^{-\frac{1}{n+1}}.
\end{align*}
\end{proposition}

\begin{proof}
Let $\tilde{R}>R$, let $\tilde{\Omega} \supset \Omega$ be open, bounded and Lipschitz regular, and consider the embedding $\tilde{i}: H_0^1(\Gamma_{\tilde{\Omega},\tilde{R}},x_{n+1}^{1-2s}) \to L^2(Q_{\tilde{\Omega},\tilde{R}},x_{n+1}^{1-2s})$. Since $\Omega$ is Lipschitz regular, there exists an extension operator (e.g., a suitable reflection and cut-off) $E:H^1(Q_{\Omega,R},x_{n+1}^{1-2s}) \to H^1(Q_{\tilde{\Omega},\tilde{R}},x_{n+1}^{1-2s})$. Let $\chi \in C_c^\infty(\Gamma_{\tilde{\Omega},\tilde{R}})$ be a smooth cutoff-function such that $\chi \equiv 1$ on $\Gamma_{\Omega,R}$ and let $m_\chi: H^1(Q_{\tilde{\Omega},\tilde{R}}, x_{n+1}^{1-2s}) \to H_0^1(\Gamma_{\tilde{\Omega},\tilde{R}}, x_{n+1}^{1-2s})$ be the multiplication operator by $\chi$. Finally, we let $R_\Omega: L^2(Q_{\tilde{\Omega},\tilde{R}},x_{n+1}^{1-2s}) \to L^2(Q_{\Omega,R},x_{n+1}^{1-2s})$, $R_\Omega \tilde{u} = \restr{\tilde{u}}{\Omega}$. Then $i = R_\Omega \circ \tilde{i} \circ m_\chi \circ E$. Consequently we have the upper bound
\begin{align*}
\sigma_k(i) \leq \Vert R_\Omega \Vert_{Op} \sigma_k(\tilde{i}) \Vert m_\chi \Vert_{Op} \Vert E \Vert_{Op} \lesssim k^{-\frac{1}{n+1}},
\end{align*}
where we have used the entropy bounds for $\tilde{i}$ from Proposition \ref{PropEstimateSingularValuesZeroBoundaryData}.
\end{proof}

\subsection{Proof of Theorem \ref{thm:low_reg_instab} for compactly supported $\bar{q}$}
\label{sec:low_reg_proof1}

With the entropy estimates from the previous subsection, we proceed to the proof of Theorem \ref{thm:low_reg_instab}. We first discuss this in the setting of $\bar{q}$ with $\supp(\bar{q}) \subset \Omega'$, for $\Omega' \subset \Omega$ compact and Lipschitz, in order to introduce the main ideas. In the next section we then extend it to more general $\bar{q}$.

\begin{proof}[Proof of Theorem \ref{thm:low_reg_instab} for compactly supported $\bar{q}$]
We split the proof into two steps. The first deduces the central compression estimate for a variable coefficient analogue of the comparison operator from Proposition \ref{prop:UniformBounds}. The second step then concludes the instability result by the usual comparison argument.

\emph{Step 1: Compression estimate.}
We follow the strategy from \cite[Chapter 5]{KRS21}. 
Let $d:= \dist(\overline{\Omega'}, \partial \Omega)>0$. For $N\in \N$ suitably determined below, we further consider a sequence of nested sets $Q_{\Omega_j,R_j}$ and $\Gamma_{\Omega_j,R_j}$ such that $\Omega = \Omega_0 \supset \Omega_1 \supset \dots \supset \Omega_N \supset \Omega_{N+1} = \Omega'$, $\dist(\overline{\Omega_j},\partial \Omega_{j-1}) \geq \frac{d}{2N}$, $R_0=d+1$ and $R_j := d+1-\frac{d}{N+1}j$, for $j\in \{1,\dots,N+1\}$.

\begin{figure}
\begin{center}
  	\begin{tikzpicture}[thick]
  	\draw[->] (-0.5,0) -- (8,0) node[right] {$\R^n$};
  	\draw[->] (0,-0.5) -- (0,5) node[left] {$x_{n+1}$};
  	\draw (-1/8,3) -- (1/8,3) node[left=4mm] {$1$};
  	\draw (-1/8,4) -- (1/8,4) node[left=4mm] {$1+d$};
  	\draw [thick,blue,decorate,decoration={brace,amplitude=7pt,mirror},yshift=-2pt](2,0) -- (6,0) node[midway,yshift=-0.5cm] {$\Omega'$};
  	\draw [thick,blue,decorate,decoration={brace,amplitude=7pt,mirror},yshift=-2pt](1,-0.7) -- (7,-0.7) node[midway,yshift=-0.5cm] {$\Omega$};
  	
  	\draw[blue] (1,0) -- (1,4) -- (7,4) node[above,right] {$Q_{\Omega_0,R_0}$} -- (7,0);
  	\draw[blue] (2,0) -- (2,3) -- (6,3) node[below=7mm,left] {$\begin{aligned} &Q_{\Omega_{N+1},R_{N+1}}\\ &\quad = Q_{\Omega',1} \end{aligned}$} -- (6,0);
  	\draw[dashed] (1.2,0) -- (1.2,3.8) -- (6.8,3.8) node[below=2mm,left] {\tiny{$Q_{\Omega_j,R_j}$}} -- (6.8,0);
  	\draw[dashed] (1.6,0) -- (1.6,3.4) -- (6.4,3.4) node[below=2mm,left] {\tiny{$Q_{\Omega_k,R_k}$}} -- (6.4,0);
  	\end{tikzpicture}
\end{center}
\caption{Nesting of the sets $Q_{\Omega_j,R_j}$ for $j\in\{0,\dots,N+1\}$, here with $k > j$.}
\end{figure}
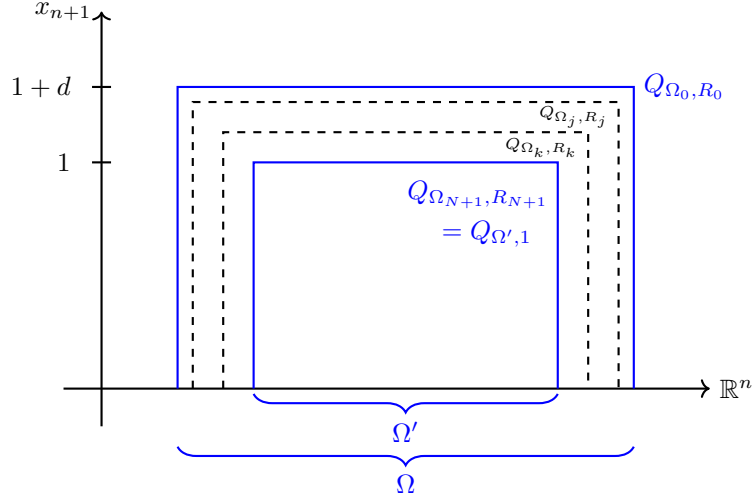

We then consider the variable coefficient analogue of the comparison operator from Proposition \ref{prop:UniformBounds} which in the variable coefficient setting reads
\begin{align*}
A: \widetilde{H}^s(W) \rightarrow H^{s}(\Omega'), \ f \mapsto u|_{\Omega'},
\end{align*}
where $u$ is a weak solution (in the Caffarelli-Silvestre extension sense) to
\begin{equation*}
\begin{cases}
\begin{alignedat}{2}
\left( \nabla \cdot a \nabla \right)^s u &= 0 \quad &&\mbox{in } \Omega,\\
u & = f \quad &&\mbox{in } \Omega_e.
\end{alignedat}
\end{cases}
\end{equation*}
We factorize this suitably to obtain the desired compression properties. Indeed, we claim that the singular values $\sigma_k(A)$ decay exponentially in $k$. More precisely, we prove that there exists some $\bar{c}>0$ such that
\begin{align}\label{ProofVariableCoeffSingularValueEstimate}
\sigma_k(A) \lesssim \exp(-\bar{c}k^{\frac{1}{n+2}}).
\end{align}
To observe this, we set
\begin{align*}
A = T_f \circ T_N \circ T_{N-1} \circ \dots \circ T_1 \circ T_{i}.
\end{align*}
Here the operators in the factorization are defined as follows:
\begin{itemize}
\item The initial operator is
\begin{align*}
T_i: \widetilde{H}^{s}(W) \rightarrow H^{1}(Q_{\Omega_0,R_0}, x_{n+1}^{1-2s}) \rightarrow L^2(Q_{\Omega_0,R_0}, x_{n+1}^{1-2s}), \
f \mapsto \tilde{u}|_{Q_{\Omega_0,R_0}} \mapsto \tilde{u}|_{Q_{\Omega_0,R_0}}.
\end{align*}
\item The iteration operators $T_1, \dots, T_N$ are given by
\begin{align*}
T_j: L^{2}&( Q_{\Omega_{j-1},R_{j-1}}, x_{n+1}^{1-2s}) \rightarrow H^1(Q_{\Omega_j,R_j}, x_{n+1}^{1-2s}) \rightarrow L^2(Q_{\Omega_j,R_j}, x_{n+1}^{1-2s}),\\
&\restr{\tilde{u}}{Q_{\Omega_{j-1},R_{j-1}}} \mapsto \restr{\tilde{u}}{Q_{\Omega_j,R_j}} \mapsto \restr{\tilde{u}}{Q_{\Omega_j,R_j}} .
\end{align*}
This uses Cacciopoli's inequality (Proposition \ref{PropCacciopoli} with $q=0$) which states that under the given boundary conditions it holds that
\begin{align*}
\| \tilde{u} \|_{H^1(Q_{\Omega_j,R_j}, x_{n+1}^{1-2s})} \leq C \frac{2N}{d} \| \tilde{u} \|_{L^2(Q_{\Omega_{j-1},R_{j-1}}, x_{n+1}^{1-2s})}. 
\end{align*}
In particular, relying on the singular value bounds for the embedding of $H^1(Q_{\Omega,R}, x_{n+1}^{1-2s}) \rightarrow L^2(Q_{\Omega,R}, x_{n+1}^{1-2s}) $ from Proposition \ref{PropEstimateSingularValues}, we obtain the singular value bounds
\begin{align*}
\sigma_k(T_j) \leq C \frac{2N}{d} k^{-\frac{1}{n+1}}.
\end{align*}
\item The final operator $T_f$ is defined as
\begin{align*}
T_f: L^2&(Q_{\Omega_N,R_N}, x_{n+1}^{1-2s}) \rightarrow H^1(Q_{\Omega',1}, x_{n+1}^{1-2s}) \rightarrow H^s(\Omega'), \\
&\tilde{u}|_{Q_{\Omega_N,R_N}} \mapsto \tilde{u}|_{Q_{\Omega',1}} \mapsto u|_{\Omega'}.
\end{align*}
\end{itemize}
Using that the entropy numbers for $T_i, T_f$ are bounded by a constant, we obtain the overall entropy estimate for the operator $A$:
\begin{align*}
\sigma_k(A) \leq C \prod\limits_{j=1}^N \sigma_{k/N}(T_j) \leq C \left( C \left(\frac{2N}{d}\right) \left(\frac{k}{N} \right)^{-\frac{1}{n+1}}\right)^N \leq C \left( \left(\frac{2C}{d}\right) \frac{N^{\frac{n+2}{n+1}}}{k^{\frac{1}{n+1}}} \right)^N.
\end{align*}
Optimizing for $N$ in terms of $k$, we choose
\begin{align*}
N = \rho k^{\frac{1}{n+2}},
\end{align*}
for $\rho>0$ a small constant. This yields that
\begin{align*}
\sigma_k(A) \leq C (C \rho^{\frac{n+2}{n+1}})^{\rho k^{\frac{1}{n+2}}},
\end{align*}
which proves the claim by rewriting the expression on the right hand side.

\emph{Step 2: Conclusion.}
As in the constant metric case, the overall instability result then follows by the comparison argument for the fractional Calder\'on argument as it is done in the proof of Theorem \ref{ThmInstability}. Indeed, we just need to take $\Gamma(q) := \Lambda_{a,q} - \Lambda_{a,0}$, write $(-\nabla\cdot a \nabla)^{s}$ instead of $(-\Delta)^{s}$ and replace the estimate \eqref{ProofFractLaplSingularValueEstimate} by estimate \eqref{ProofVariableCoeffSingularValueEstimate}.
\end{proof}

For the convenience of the reader, we recall from \cite[Proposition 2.3]{JLX14} Caccioppoli's estimate with our boundary conditions. 

\begin{proposition}[Proposition 2.3 in \cite{JLX14}]
\label{PropCacciopoli}
Let $s\in(0,1)$, $R>0$, $\Omega\subset \R^n$ be open, non-empty, bounded and Lipschitz. Let $\tilde{a} \in L^{\infty}(\R^{n+1}_+, \R^{(n+1) \times (n+1)})$ be symmetric and uniformly elliptic with ellipticity constant $\lambda$. Let $q \in L^{p}(\Omega)$ with $p> \frac{n}{2s}$, or $q \in L^{\frac{n}{2s}}(\Omega)$ with $\|q\|_{L^{\frac{n}{2s}}(\Omega)} \leq \theta$ and $\theta=\theta(n,s,\lambda) \in (0,1)$ sufficiently small. Assume that $\tilde{u} \in H^1(Q_{\Omega,R}, x_{n+1}^{1-2s})$ satisfies
\begin{equation*}
\begin{cases}
\begin{alignedat}{2}
\nabla \cdot x_{n+1}^{1-2s} \tilde{a} \nabla \tilde{u} &= 0 \quad &&\text{in } Q_{\Omega,R},\\
\lim_{x_{n+1}\to0} x_{n+1}^{1-2s} \partial_{n+1} \tilde{u} &= q \tilde{u}(\cdot,0) \quad &&\text{on } \Omega \times \{0\},
\end{alignedat}
\end{cases}
\end{equation*}
in the weak sense, i.e.,
\begin{align}\label{EqWeakFormCSNoBoundaryData}
\int_{Q_{\Omega,R}} x_{n+1}^{1-2s} \nabla \tilde{\varphi} \cdot \tilde{a} \nabla \tilde{u} \ d (x,x_{n+1}) = \int_{\Omega} q (\tilde{u} \tilde{\varphi})(x,0) dx \quad \text{for all } \tilde{\varphi} \in H_0^1(\Gamma_{\Omega,R},x_{n+1}^{1-2s}).
\end{align}
Then there exists $C>0$ such that for all $x_0 \in \Gamma_{\Omega,R}$ and for all $r_1,r_2>0$ such that $0 < r_1 < r_2 < \dist \left( x_0, \ \left( \partial\Omega \times (0,R) \right) \cup \left( \Omega \times \{R\} \right) \right)$, the following estimate holds
\begin{align*}
\int_{B_{r_1}(x_0) \cap \R^{n+1}_+} x_{n+1}^{1-2s} \vert \nabla \tilde{u} \vert^2 d(x,x_{n+1}) \leq \frac{C}{(r_2 - r_1)^2} \int_{B_{r_2}(x_0) \cap \R^{n+1}_+} x_{n+1}^{1-2s} \vert \tilde{u} \vert^2 d(x,x_{n+1}).
\end{align*}
The constant $C$ only depends on the dimension $n$, $s$, $\Vert \tilde{a} \Vert_{L^\infty(Q_{\Omega,R})}$, $\|q \|_{L^{p}(\Omega)}$, $\diam (Q_{\Omega,R})$ and $\lambda$.
\end{proposition}

\begin{proof}
If $q \in L^p(\Omega)$ for $p>\frac{n}{2s}$ we can choose $\rho>0$ small such that $\Vert q \Vert_{L^{\frac{n}{2s}}(B_\rho(x_0))} < \theta$ for $\theta = \theta(n,s,\lambda)\in(0,1)$ sufficiently small. In that case scale $r_1$ and $r_2$ by $\alpha = \alpha(n, \Vert q \Vert_{L^p(\Omega)}, \theta, \diam(Q_{\Omega,R})) \in(0,1]$ such that additionally $0 < \alpha r_1 < \alpha r_2 < \rho$. Let $\tilde{\eta} \in C_c^\infty(\Gamma_{\Omega,R})$ be a cut-off function such that 
\begin{align*}
0 \leq \tilde{\eta} \leq 1, \quad \tilde{\eta} \equiv 1 \text{ on } B_{r_1}(x_0) \cap \overline{\R^{n+1}_+}, \quad \supp(\tilde{\eta}) \Subset B_{r_2}(x_0) \cap \overline{\R^{n+1}_+}, \quad \vert \nabla \tilde{\eta} \vert \lesssim \frac{1}{r_2-r_1}.
\end{align*}
We easily verify that in this case $\tilde{\eta}^2 \tilde{u} \in H_0^1(\Gamma_{\Omega,R}, x_{n+1}^{1-2s})$. Then, testing equation \eqref{EqWeakFormCSNoBoundaryData} against $\tilde{\eta}^2 \tilde{u}$ we get
\begin{equation}\label{Eq1PropCacciopoliIneq}
\begin{aligned}
\int_{\Omega} q (\tilde{\eta}\tilde{u})^2(x,0) dx &= \int_{Q_{\Omega,R}} x_{n+1}^{1-2s} \nabla (\tilde{\eta}^2\tilde{u}) \cdot \tilde{a} \nabla \tilde{u} \ d(x,x_{n+1})\\
&= \int_{Q_{\Omega,R}} x_{n+1}^{1-2s} (2\tilde{\eta}\nabla\tilde{\eta}\tilde{u} + \tilde{\eta}^2\nabla\tilde{u}) \cdot \tilde{a} \nabla\tilde{u} \ d(x,x_{n+1}).
\end{aligned}
\end{equation}
Using the ellipticity of $\tilde{a}$, equation \eqref{Eq1PropCacciopoliIneq}, the Cauchy-Schwarz and Young's inequality, we deduce
\begin{equation}\label{Eq2PropCacciopoliIneq}
\begin{aligned}
\lambda \Vert \tilde{\eta} \nabla\tilde{u} &\Vert^2_{L^2(Q_{\Omega,R},x_{n+1}^{1-2s})} = \lambda \int_{Q_{\Omega,R}} x_{n+1}^{1-2s} \tilde{\eta}^2  \vert \nabla\tilde{u} \vert^2 d(x,x_{n+1}) \leq \int_{Q_{\Omega,R}} x_{n+1}^{1-2s} \tilde{\eta}^2 \nabla\tilde{u} \cdot \tilde{a} \nabla\tilde{u} \ d(x,x_{n+1})\\
&= - \int_{Q_{\Omega,R}} x_{n+1}^{1-2s} 2\tilde{\eta} \nabla\tilde{\eta} \tilde{u} \cdot \tilde{a} \nabla\tilde{u} \ d(x,x_{n+1}) \ + \int_{\Omega} q (\tilde{\eta}\tilde{u})^2(x,0) dx\\
&\leq 2 \Vert \tilde{a} \Vert_{L^\infty(Q_{\Omega,R})} \Vert \tilde{\eta} \nabla\tilde{u} \Vert_{L^2(Q_{\Omega,R},x_{n+1}^{1-2s})} \Vert \tilde{u} \nabla\tilde{\eta} \Vert_{L^2(Q_{\Omega,R},x_{n+1}^{1-2s})} + \int_{\Omega} \vert q \vert (\tilde{\eta}\tilde{u})^2(x,0) dx\\
&\leq \frac{\lambda}{2} \Vert \tilde{\eta} \nabla\tilde{u} \Vert_{L^2(Q_{\Omega,R},x_{n+1}^{1-2s})}^2 + \frac{(2\Vert \tilde{a} \Vert_{L^\infty(Q_{\Omega,R})})^2}{2\lambda} \Vert \tilde{u} \nabla\tilde{\eta} \Vert_{L^2(Q_{\Omega,R},x_{n+1}^{1-2s})}^2 + \int_{\Omega} \vert q \vert (\tilde{\eta}\tilde{u})^2(x,0) dx\\
\end{aligned}
\end{equation}
Additionally, by using that $\tilde{\eta}$ is supported on $B_{r_2}(x_0)$, Hölder's inequality, Sobolev's inequality and the trace estimate we have
\begin{equation}\label{Eq3PropCacciopoliIneq}
\begin{aligned}
\int_\Omega \vert q \vert (\tilde{\eta}\tilde{u})^2(x,0) dx &= \int_{\Omega \cap B_{r_2}(x_0)} \vert q \vert (\tilde{\eta}\tilde{u})^2(x,0) dx \leq \Vert q \Vert_{L^{\frac{n}{2s}}(\Omega \cap B_{r_2}(x_0))} \Vert (\tilde{\eta}\tilde{u})(\cdot,0) \Vert_{L^{\frac{2n}{n-2s}}(\Omega)}^2\\
&\leq \theta C(n,s) \Vert (\tilde{\eta}\tilde{u})(\cdot,0) \Vert_{H^s(\Omega)}^2 \leq \theta C(n,s) \Vert \nabla (\tilde{\eta}\tilde{u}) \Vert_{L^2(Q_{\Omega,R},x_{n+1}^{1-2s})}^2\\
&\leq \theta C(n,s) \left( \Vert \tilde{\eta} \nabla\tilde{u} \Vert_{L^2(Q_{\Omega,R},x_{n+1}^{1-2s})}^2 + \Vert \tilde{u}\nabla\tilde{\eta} \Vert_{L^2(Q_{\Omega,R}, x_{n+1}^{1-2s})}^2 \right)
\end{aligned}
\end{equation}
If $\theta>0$ is so small that $\theta C(n,s) \leq \frac{\lambda}{4}$, then by inserting \eqref{Eq3PropCacciopoliIneq} into \eqref{Eq2PropCacciopoliIneq}, rearranging terms and using the definition of $\tilde{\eta}$ we finally infer
\begin{align*}
\int_{B_{r_1}(x_0) \cap \R^{n+1}_+} x_{n+1}^{1-2s} \vert \nabla\tilde{u} \vert^2 d(x,x_{n+1}) \leq \frac{C}{(r_2-r_1)^2} \int_{B_{r_2}(x_0) \cap \R^{n+1}_+} x_{n+1}^{1-2s} \vert \tilde{u} \vert^2 d(x,x_{n+1})
\end{align*}
with $C = C(n, s, \Vert \tilde{a} \Vert_{L^\infty(Q_{\Omega,R})}, \Vert q \Vert_{L^p(\Omega)}, \diam(Q_{\Omega,R}), \lambda)$, which finishes the proof. 
\end{proof}

\subsection{Proof of Theorem \ref{thm:low_reg_instab} for general background potential $\bar{q}$}
\label{sec:low_reg_proof2}

Since we will make use of this in the later sections, we here extend our results from the Section \ref{sec:low_reg_proof1} to include background potentials with non-compact support. The main difference here occurs in the definition of the comparison operator which incorporates the background potential.

\begin{proof}
Compared to the previous result, we seek to work with the comparison operator $A_{\bar{q}}: \widetilde{H}^{s}(W) \rightarrow H^{s}(\Omega')$ defined by $f \mapsto u_0|_{\Omega'}$, here $u_0 \in H^{s}(\R^n)$ denotes a solution to
\begin{equation*}
\begin{cases}
\begin{alignedat}{2}
(-\nabla \cdot a  \nabla)^s u_0 + \bar{q} u_0 &= 0 \quad &&\mbox{in } \Omega,\\
u_0 & = f \quad &&\mbox{in } \Omega_e.
\end{alignedat}
\end{cases}
\end{equation*}
Moreover, we let $u$ be a solution to the same equation with potential $q \in K := \{q \in L^p(\Omega): q \in \bar{q} + B_{r_0}^{W_0^{\delta, \frac{n}{2s}}(\Omega')}(0)\}$.
The difference $w:= u - u_0$ then satisfies the equation
\begin{equation*}
\begin{cases}
\begin{alignedat}{2}
\left( (-\nabla \cdot a \nabla)^s + q \right) w  &= (\bar{q} - q)u_0 \quad &&\mbox{in } \Omega,\\
w & = 0 \quad &&\mbox{in } \Omega_e,
\end{alignedat}
\end{cases}
\end{equation*}
and hence, doing the estimates in a similar fashion as in \eqref{EqEstimateForComparisonOperator} using that $\supp(\bar{q} - q) \subset \Omega'$,
\begin{align*}
\|\Gamma_{\bar{q}}(q) f\|_{H^{-s}(W)} \leq C \|u_0\|_{H^{s}(\Omega')} =  C \|A_{\bar{q}}(f)\|_{H^{s}(\Omega')},
\end{align*}
where $\Gamma_{\bar{q}}(q):= \Lambda_{q}- \Lambda_{\bar{q}}$. 
As in Section \ref{sec:low_reg_proof1} the claimed instability result then follows from compression estimates for the comparison operator $A_{\bar{q}}$.

The remainder of the argument then follows from an analogous singular value estimate as in the proof presented in the previous section, using the smallness of the background potential $\bar{q}$ and the general Caccioppoli inequality from Proposition \ref{Eq1PropCacciopoliIneq} above.

\end{proof}

\section{Instability for the fractional conductivity equation}\label{SectionProofCondEq}

Building on the results for the fractional Calder\'on problem in its Schrödinger formulation, in this section we proceed to the study of the uniqueness and instability properties of the problem
\begin{align}\label{before-reduction}
\begin{cases}
\begin{alignedat}{2}
\left( (\nabla \cdot)^s (\Theta \cdot\nabla^s) + q \right) u &= 0 \qquad &&\text{in } \Omega,\\
u &= f \qquad &&\text{in }  \Omega_e.
\end{alignedat}
\end{cases}
\end{align}

We start by proving Proposition \ref{uniqueness-frac-cond}, which shows uniqueness.

\begin{proof}[Proof of Proposition \ref{uniqueness-frac-cond}]
    Our proof is based on the fractional Liouville reduction (see \cite{C20} for the isotropic case, and \cite{C22} for the anisotropic one). We have that
\begin{align*}
    (\nabla \cdot)^s (\Theta \cdot\nabla^su) + qu &= \gamma^{1/2}(-\Delta)^sw + \gamma^{1/2}q_\gamma w + \gamma^{-1/2}qw 
    \\ & = 
    \gamma^{1/2}\left((-\Delta)^sw + Q w\right),
\end{align*}
where $w:= \gamma^{1/2}u$, $q_\gamma := -\gamma^{-1/2}(-\Delta)^s(\gamma^{1/2}-1)$, and $Q := q_\gamma + \gamma^{-1}q$. The second equivalence statement of what follows is more general than what is needed in this proof. However we state it in this form for later use (see proof of Theorem \ref{ThmInstability_other} below). Let $p\in[\frac{n}{2s},\infty)$. By our assumptions on $\gamma$, we see that $q_\gamma\in L^{p}(\R^n)$ and 
\begin{equation}\label{regularity-gamma}
    \gamma^{1/2}u\in H^s(\R^n) \Leftrightarrow u\in H^s(\R^n), \qquad q\in W^{\delta,p}(\Omega) \Leftrightarrow Q\in W^{\delta,p}(\Omega),
\end{equation} 
for all $\delta \geq 0$. In fact, since $\gamma$ is smooth, bounded, and strictly larger than a positive constant, it follows that $\gamma^{r}$ is a multiplier on $H^s(\R^n)$ for all $r\in \R$. This proves the first statement in \eqref{regularity-gamma} by choosing $r=\pm 1/2$. Moreover, since $\gamma^{1/2}-1 \in C^\infty_c(\R^n)$ by assumption, in particular we have $(-\Delta)^{s}(\gamma^{1/2}-1 )\in W^{\delta,p}(\R^n)$  for all $\delta \geq 0$ by the mapping properties of $(-\Delta)^s$. Since $C^1(\R^n)$-functions are multipliers of $W^{\delta,p}(\R^n)$ for $\delta\in[0,1)$, we have $q_\gamma\in W^{\delta,p}(\R^n)$, which in particular implies that $q_\gamma\in L^{p}(\R^n)$ as claimed. Finally, from the fact that $\gamma,\gamma^{-1}\in C^\infty(\R^n)$ we deduce that $q,Q$ must have the same regularity, which proves the second part of \eqref{regularity-gamma}.

By the fractional Liouville reduction, $u$ solves \eqref{before-reduction} if and only if $w$ solves
\begin{align}\label{after-reduction}
\begin{cases}
\begin{alignedat}{2}
((-\Delta)^s + Q) w &= 0 \qquad &&\text{in } \Omega,\\
w &= \gamma^{1/2}f \qquad &&\text{in }  \Omega_e,
\end{alignedat}
\end{cases}
\end{align}
where $\gamma^{1/2}f\in H^s(\R^n)$. Let $u_f, w_f$ be the solutions of problems \eqref{before-reduction}, \eqref{after-reduction} with exterior values $f, \gamma^{1/2}f$ respectively. We seek to compare the Dirichlet-to-Neumann maps $\Lambda_{\gamma,q}, \Lambda_{Q}$ corresponding to the problems \eqref{before-reduction}, \eqref{after-reduction} in the cases of two distinct potentials $q_1,q_2$, and thus we let $Q_{j} := q_\gamma + \gamma^{-1}q_j$ for $j=1,2$. We can compute
\begin{equation}\label{EqReductionInnerProductEquivalence}
\begin{aligned}
    \langle (\Lambda_{Q_{1}}-\Lambda_{Q_{2}})(\gamma^{1/2}f_1), \gamma^{1/2}f_2 \rangle & = \int_\Omega (Q_{1}-Q_{2})w_{f_1} w_{f_2} dx
    \\ & =
    \int_\Omega \gamma^{-1}(q_1-q_2)(\gamma^{1/2}u_{f_1}) (\gamma^{1/2}u_{f_2}) \, dx
    \\ & =
    \int_\Omega (q_1-q_2)u_{f_1} \, u_{f_2} \, dx
    \\ & =
    \langle (\Lambda_{\gamma,q_1}-\Lambda_{\gamma,q_2})f_1, f_2 \rangle ,
\end{aligned}
\end{equation}
and thus
\begin{equation}\label{second-norm-equivalence}
\begin{split}
   \| \Lambda_{\gamma,q_1}-\Lambda_{\gamma,q_2} \|_{\widetilde H^s(W_1)\rightarrow H^{-s}(W_2)} & = \sup_{f_j\in \widetilde H^s(W_j)} \frac{ |\langle (\Lambda_{\gamma,q_1}-\Lambda_{\gamma,q_2})f_1, f_2 \rangle| }{\|f_1\|_{H^s(\R^n)}\|f_2\|_{H^s(\R^n)}} \\ 
   &\sim \sup_{f_j\in \widetilde H^s(W_j)} \frac{ |\langle (\Lambda_{Q_{1}}-\Lambda_{Q_{2}})(\gamma^{1/2}f_1), \gamma^{1/2}f_2 \rangle| }{\|\gamma^{1/2} f_1\|_{H^s(\R^n)}\| \gamma^{1/2}f_2\|_{H^s(\R^n)}} 
   \\ & = 
   \|\Lambda_{Q_{1}}-\Lambda_{Q_{2}} \|_{\widetilde H^s(W_1)\rightarrow H^{-s}(W_2)},
\end{split}
\end{equation}
because $\|f\|_{H^s(\R^n)} \sim \|\gamma^{1/2} f\|_{H^s(\R^n)}$ for all $f\in H^s(\R^n)$, and $f\in \widetilde H^s(W) \Leftrightarrow \gamma^{1/2} f\in \widetilde H^s(W)$. Therefore, by the assumption on the Dirichlet-to-Neumann maps we have $\|\Lambda_{Q_{1}}-\Lambda_{Q_{2}} \|_{\widetilde H^s(W_1)\rightarrow H^{-s}(W_2)}=0$, that is $$\langle (\Lambda_{Q_{1}}-\Lambda_{Q_{2}}) f_1, f_2 \rangle=0, \qquad \mbox{ for all } f_j\in \widetilde H^s(W_j),\; j=1,2.$$
By the results of \cite{RS20} it now immediately follows that $Q_1=Q_2$ in $\Omega$. Since $\gamma$ is fixed a priori, from this we are able to deduce that $q_1=q_2$ in $\Omega$.
\end{proof}

\begin{remark}
    In the case in which the conductivity $\gamma$ is not fixed a priori, the relation $Q_1=Q_2$ in $\Omega$ only defines a gauge class to which the couples $(\gamma_1, q_1)$, $(\gamma_2, q_2)$ both belong. This was to be expected in light of \cite[Section 6]{C20a}.
\end{remark}

With the proof of Proposition \ref{uniqueness-frac-cond} at hand, we now move to the proof of our main instability result  for the fractional conductivity equation from Theorem \ref{ThmInstability_other}.

\begin{proof}[Proof of Theorem \ref{ThmInstability_other}] The proof is again based on the fractional Liouville reduction, and thus we continue using the symbols from the previous proof. Observe that since $\delta\in(0,1)$, $\gamma^r$ is a multiplier of $W^{\delta,p}(\Omega)$ for all $r\in\R$, so we can estimate
\begin{align*}
\|Q_1-Q_2\|_{W^{\delta,p}(\Omega)} = \|\gamma^{-1}(q_1 -q_2)\|_{W^{\delta,p}(\Omega)} \lesssim \|q_1 -q_2\|_{W^{\delta,p}(\Omega)}.
\end{align*}
Similarly, we obtain $\|q_1-q_2\|_{W^{\delta,p}(\Omega)} = \|\gamma(Q_1 -Q_2)\|_{W^{\delta,p}(\Omega)} \lesssim \|Q_1 -Q_2\|_{W^{\delta,p}(\Omega)} $, and thus
\begin{equation}
    \label{first-norm-equivalence}
    \|q_1-q_2\|_{W^{\delta,p}(\Omega)} \sim \|Q_1 -Q_2\|_{W^{\delta,p}(\Omega)}.
\end{equation} 

The proof now follows from \eqref{first-norm-equivalence} and \eqref{second-norm-equivalence}. In order to see this, let $\bar{q} \in L^{p}(\Omega)$ satisfy condition (Aq) with respect to problem \eqref{before-reduction} and radius $r_0>0$, and define $\bar Q := q_\gamma + \gamma^{-1}\bar q$. Then, $\bar Q\in L^{p}(\Omega)$ by \eqref{regularity-gamma}, and it satisfies condition (Aq) with respect to problem \eqref{after-reduction} and radius $R_0:=\|\gamma\|_{C^1(\R^n)}^{-1} r_0>0$. In fact, if $\| Q-\bar Q \|_{L^{p}(\Omega)}<R_0$, then
$$ \|(\gamma Q-\gamma q_\gamma) - \bar q\|_{L^{\frac{n}{2s}}(\Omega)} = \|\gamma Q - \gamma \bar Q\|_{L^{\frac{n}{2s}}(\Omega)} < \overline\gamma \|\gamma\|_{C^1(\R^n)}^{-1} r_0 \leq r_0,$$
which implies that $0$ is not a Dirichlet eigenvalue for problem \eqref{before-reduction} with potential $q:= \gamma Q-\gamma q_\gamma$. Equivalently, $0$ is not a Dirichlet eigenvalue for the transformed problem \eqref{after-reduction} with potential $Q$, which shows condition (Aq) in the reduced case.

Let now $Q_1, Q_2 \in L^p(\Omega)$ be such that $Q_1, Q_2 \in \bar{Q} + B_{R_0}^{W_0^{\delta,p}(\Omega')}(0)$ with
\begin{align*}
\Vert \Lambda_{Q_1} - \Lambda_{Q_2} \Vert_{\tilde{H}^s(W) \to H^{-s}(W)} &\leq \exp\left(-c E^{-\frac{1}{\delta(2+\frac{5}{n})}}\right),\\
\Vert Q_1 - Q_2 \Vert_{L^{p}(\Omega)} &\geq E,
\end{align*} 
and $0<E<R_0$, with $E$ small enough. The existence of $Q_1, Q_2$ follows from Theorem \ref{thm:low_reg_instab}. Using \eqref{first-norm-equivalence} and \eqref{second-norm-equivalence} we deduce
\begin{align*}
\Vert \Lambda_{\gamma,q_1} - \Lambda_{\gamma,q_2} \Vert_{\tilde{H}^s(W) \to H^{-s}(W)} &\lesssim \exp\left(-c E^{-\frac{1}{\delta(2+\frac{5}{n})}}\right),\\
\Vert q_1 - q_2 \Vert_{L^{p}(\Omega)} &\gtrsim E,
\end{align*} 
where $q_j := \gamma Q_j - \gamma q_\gamma$, for $j=1,2$. We have $\supp(q_j-\bar{q})\subset\Omega'$ and, by \eqref{regularity-gamma},  $q_j \in L^p(\Omega)$. It also holds
$$\|q_j-\bar q\|_{W^{\delta,p}(\Omega)} \leq \|\gamma\|_{C^1(\Omega)}\|Q_j-\bar Q\|_{W^{\delta,p}(\Omega)} < R_0\|\gamma\|_{C^1(\R^n)} = r_0,$$
and thus $q_1, q_2 \in \bar{q} + B_{r_0}^{W_0^{\delta,p}(\Omega')}(0)$. Finally, we see that the statement of the theorem holds with $\varepsilon := CE\|\gamma\|_{C^1(\R^n)}$.
\end{proof}

\section{Single measurement instability}\label{sec:single_meas}

In this section we provide the proof of Corollaries \ref{CorSingleMeasFractCald} and \ref{CorSingleMeasFractCond} as an immediate consequence of the main results of Theorem \ref{ThmInstability} and Theorem \ref{ThmInstability_other}, respectively.

\begin{proof}[Proof of Corollary \ref{CorSingleMeasFractCald}]
Let $q_1,q_2 \in L^{\frac{n}{2s}}(\Omega)$ be two potentials having the properties from Theorem \ref{ThmInstability}. Then we have
\begin{align*}
\Vert \Lambda_{q_1}(f) - \Lambda_{q_2}(f) \Vert_{H^{-s}(W)} \leq \Vert \Lambda_{q_1} - \Lambda_{q_2} \Vert_{\widetilde{H}^s(W) \to H^{-s}(W)} \Vert f \Vert_{H^s(W)} \lesssim \exp\left(-c \varepsilon^{- \frac{1}{\delta(2+\frac{1}{n})}}\right)
\end{align*}
and the result follows.
\end{proof}

The proof of Corollary \ref{CorSingleMeasFractCond} is done in the same way applying Theorem \ref{ThmInstability_other} instead of Theorem \ref{ThmInstability}.

We conclude this section with the proof of the single measurement stability result for the fractional conductivity problem.

\begin{proof}[Proof of Proposition \ref{PropSingleMeasStabFractCond}]
This result is a consequence of the fractional Liouville reduction and an application of Theorem 1 from \cite{R21}. We will again use the notation from the proof of Proposition \ref{uniqueness-frac-cond} in Section \ref{SectionProofCondEq}.\\
We recall that, by the fractional Liouville reduction, $u$ solves
\begin{align*}
\begin{cases}
\begin{alignedat}{2}
\left( (\nabla \cdot)^s (\Theta \cdot\nabla^s) + q \right) u &= 0 \qquad &&\text{in } \Omega,\\
u &= f \qquad &&\text{in }  \Omega_e,
\end{alignedat}
\end{cases}
\end{align*}
if and only if $w$ solves
\begin{align*}
\begin{cases}
\begin{alignedat}{2}
((-\Delta)^s + Q) w &= 0 \qquad &&\text{in } \Omega,\\
w &= \gamma^{1/2}f \qquad &&\text{in }  \Omega_e,
\end{alignedat}
\end{cases}
\end{align*}
where $w := \gamma^{1/2} u$, $Q := q_\gamma + \gamma^{-1}q$ and $q_\gamma := -\gamma^{-1/2} (-\Delta)^s \left( \gamma^{1/2} - 1 \right)$.\\
First of all we note that
\begin{align*}
\frac{\Vert \gamma^{1/2} f \Vert_{H^s(W)}}{\Vert \gamma^{1/2} f \Vert_{L^2(W)}} \leq \frac{\Vert \gamma^{1/2} \Vert_{C^1(\R^n)} \Vert f \Vert_{H^s(W)}}{\underline{\gamma}^{1/2} \Vert f \Vert_{L^2(W)}} \leq C \frac{\Vert f \Vert_{H^s(W)}}{\Vert f \Vert_{L^2(W)}} \leq \tilde{F},
\end{align*}
and
\begin{align*}
\Vert Q_j \Vert_{C^{0,s}(\Omega)} &\leq \Vert \gamma^{-1/2} (-\Delta)^s \left( \gamma^{1/2}-1 \right) \Vert_{C^{0,s}(\Omega)} + \Vert \gamma^{-1} q_j \Vert_{C^{0,s}(\Omega)}\\
&\leq \Vert \gamma^{-1/2} (-\Delta)^s \left( \gamma^{1/2}-1 \right) \Vert_{C^1(\Omega)} + \Vert \gamma^{-1} \Vert_{C^{0,s}(\Omega)} \Vert q_j \Vert_{C^{0,s}(\Omega)} \leq \tilde{E},
\end{align*}
where $\tilde{F}$ and $\tilde{E}$ only depend on $\gamma$ and $F$, $E$, respectively. Here we used that, since $\gamma^{1/2} -1 \in C_c^\infty(\R^n)$ by assumption, we obtain by the mapping properties of $(-\Delta)^s$ and Sobolev embedding that $q_\gamma \in C^m(\R^n)$ for all $m\in\N$. Moreover we have $\supp(Q_2 - Q_1) \subset \Omega'$. Thus, we can apply Theorem 1 from \cite{R21} (see also Remark 5.2 in \cite{R21}) to deduce
\begin{align}\label{EqStabilityEstimate}
\Vert Q_1 - Q_2 \Vert_{L^\infty(\Omega)} \leq \tilde{\omega} (\Vert \Lambda_{Q_1}(\gamma^{1/2}f) - \Lambda_{Q_2}(\gamma^{1/2}f) \Vert_{H^{-s}(W)})
\end{align}
for some $\tilde{\omega}$ with $\tilde{\omega}(t) \leq \tilde{c} \vert \log(\tilde{c}t) \vert^{-\tilde \mu}$, where $\tilde\mu$ and $\tilde{c}$ depend on the given parameters.\\
As in \eqref{EqReductionInnerProductEquivalence}, we derive for $h \in \widetilde{H}^s(W)$
\begin{align*}
\langle (\Lambda_{Q_{1}}-\Lambda_{Q_{2}})(\gamma^{1/2}f), \gamma^{1/2}h \rangle = \langle (\Lambda_{\gamma,q_1}-\Lambda_{\gamma,q_2})f, h \rangle,
\end{align*}
and thus
\begin{equation}\label{EqEquivalenceOfNorms}
\begin{aligned}
\Vert \Lambda_{\gamma,q_1}(f)& - \Lambda_{\gamma,q_2}(f) \Vert_{H^{-s}(W)} = \sup_{h \in \widetilde{H}^s(W)} \frac{\vert \langle (\Lambda_{\gamma,q_1} - \Lambda_{\gamma,q_2})(f), h \rangle \vert}{\Vert h \Vert_{H^s(W)}}\\
&\sim \sup_{h \in \widetilde{H}^s(W)} \frac{\vert \langle (\Lambda_{Q_{1}}-\Lambda_{Q_{2}})(\gamma^{1/2}f), \gamma^{1/2}h \rangle \vert}{\Vert \gamma^{1/2} h \Vert_{H^s(W)}} = \Vert \Lambda_{Q_1}(\gamma^{1/2}f) - \Lambda_{Q_2}(\gamma^{1/2}f) \Vert_{H^{-s}(W)},
\end{aligned}
\end{equation}
since $\Vert h \Vert_{H^s(\R^n)} \sim \Vert \gamma^{1/2} h \Vert_{H^s(R^n)}$ for all $h \in H^s(\R^n)$, and $h \in \widetilde{H}^s(W)$ if and only if $\gamma^{1/2}h \in \widetilde{H}^s(W)$. Additionally, we have
\begin{align*}
\Vert Q_1 - Q_2 \Vert_{L^\infty(\Omega)} = \Vert \gamma^{-1} (q_1 - q_2) \Vert_{L^\infty(\Omega)} \leq \underline{\gamma}^{-1} \Vert q_1 - q_2 \Vert_{L^\infty(\Omega)},
\end{align*}
and
\begin{align*}
\Vert q_1 - q_2 \Vert_{L^\infty(\Omega)} = \Vert \gamma (Q_1 - Q_2) \Vert_{L^\infty(\Omega)} \leq \overline{\gamma} \Vert Q_1 - Q_2 \Vert_{L^\infty(\Omega)},
\end{align*}
and thus
\begin{align}\label{EqEquivalenceOfNorms2}
\Vert Q_1 - Q_2 \Vert_{L^\infty(\Omega)} \sim \Vert q_1 - q_2 \Vert_{L^\infty(\Omega)}.
\end{align}
Putting together \eqref{EqStabilityEstimate}, \eqref{EqEquivalenceOfNorms} and \eqref{EqEquivalenceOfNorms2}, we finally infer
\begin{align*}
\Vert q_1 - q_2 \Vert_{L^\infty(\Omega)} \leq \omega (\Vert \Lambda_{\gamma,q_1}(f) - \Lambda_{\gamma,q_2}(f) \Vert_{H^{-s}(W)})
\end{align*}
for some $\omega$ with $\omega(t) \leq c \vert \log(ct) \vert^{-\mu}$, where $\mu$ and $c$ depend on the given parameters.
\end{proof}

\section*{Acknowledgements}

H.B. and A.R. gratefully acknowledge support of the Hausdorff Center for Mathematics which is funded by the Deutsche Forschungsgemeinschaft (DFG, German Research Foundation) under Germany's Excellence Strategy – EXC-2047/1 – 390685813. G.C. was supported by the PDE-Inverse project of the European Research Council of the European Union and Research Council of Finland FAME-flagship and the grant 336786. Views and opinions expressed are those of the authors only and do not necessarily reflect those of the European Union or the other funding organizations. Neither the European Union nor the other funding organizations can be held responsible for them.

\appendix

\section{Analyticity estimates for the heat equation}

For convenience of the reader, we recall the following analyticity estimates for the heat equation. The idea of the proof is based on \cite[Lemma 2.2]{EMZ17}.

\begin{lemma}\label{Lemma2.2[Emz17]}
Let $\Omega$ be an open, bounded, Lipschitz set and let $\Omega' \Subset \Omega$ be open and Lipschitz. Let $w$ be a weak solution to
\begin{equation*}
\begin{cases}
\begin{alignedat}{2}
(\p_t - \Delta) w &= 0 \quad &&\mbox{in } \Omega \times (0,1),\\
w & = 0 \quad &&\mbox{on } \Omega \times \{0\}.
\end{alignedat}
\end{cases}
\end{equation*}
Then for all $\gamma\in\N^n$,
\begin{align*}
\Vert \partial_t \partial_x^\gamma w \Vert_{L^2(\Omega' \times (0,1))} + \sum\limits_{\alpha\in \N^n, |\alpha|\leq 2}\|\partial_x^{\alpha} \partial_x^{\gamma} w \|_{L^2(\Omega' \times (0,1))} \leq N \|w\|_{L^2(\Omega \times [0,1])} \rho^{-|\gamma|}|\gamma|!
\end{align*}
for some $N = N(n, \Omega', \Omega)$ and some $\rho = \rho(n, \Omega', \Omega)\in (0,1)$.
\end{lemma}

\begin{proof}
We prove the statement for balls $\Omega = B_R$ and $\Omega' = B_r$ with $r < R \leq 1$. A typical ball covering argument then yields the result for general domains $\Omega$ and $\Omega'$. As in \cite{EMZ17} we prove the statement for balls inductively. Our induction hypothesis then reads
\begin{equation}\label{EqInductionHypothesisLemma2.2[EMZ17]}
\begin{aligned}
(R-r)^2 \Vert \partial_t \partial_x^\gamma w \Vert_{L^2(B_r \times (0,1))} &+ \sum_{\alpha\in\N^n, \vert\alpha\vert\leq2} (R-r)^{\vert\alpha\vert} \Vert \partial_x^\alpha \partial_x^\gamma w \Vert_{L^2(B_r \times (0,1))}\\
&\leq N_1 \Vert w \Vert_{L^2(B_R \times (0,1))} \left( \frac{N_2}{R-r} \right)^{\vert\gamma\vert} \vert\gamma\vert! \qquad \mbox{for any } \gamma\in\N^n
\end{aligned}
\end{equation}
for $N_1$ and $N_2$ large enough constants to be chosen. We define
\begin{align*}
D := 2K+2,
\end{align*}
where $K$ is the constant in the Gagliardo-Nirenberg interpolation inequality (see below). We now start with the induction.

\textit{Step 1: Initial step.} When $\vert\gamma\vert = 0$, then by Lemma A.5 in \cite{EMZ17} we have for $\delta>0$ such that $r+\delta \leq R$ for some constant $N_0\geq1$ depending only on the dimension $n$ and $R$
\begin{align*}
\Vert \partial_t w \Vert_{L^2(B_r \times (0,1))} + \Vert D^2_x w \Vert_{L^2(B_r \times (0,1))} \leq N_0 \delta^{-2} \Vert w \Vert_{L^2(B_{r+\delta} \times (0,1))}
\end{align*}
and thus by the Gagliardo-Nirenberg interpolation inequality we have for appropriately chosen $N_1$ depending on $K$ (see also below when the Gagliardo-Nirenberg interpolation inequality is used again)
\begin{align}\label{Eq1Lemma2.2[EMZ17]}
\delta^2 \Vert \partial_t w \Vert_{L^2(B_r \times (0,1))} + \sum_{\alpha\in\N^n, \vert\alpha\vert\leq2} \delta^{\vert\alpha\vert} \Vert \partial_x^\alpha w \Vert_{L^2(B_r \times (0,1))} \leq N_1 \Vert w \Vert_{L^2(B_{r+\delta} \times (0,1))}.
\end{align}
If we choose $\delta = R-r$ we get
\begin{align}\label{Eq10Lemma2.2[EMZ17]}
(R-r)^2 \Vert \partial_t w \Vert_{L^2(B_r \times (0,1))} +  \sum_{\alpha\in\N^n, \vert\alpha\vert\leq2} (R-r)^{\vert\alpha\vert} \Vert \partial_x^\alpha w \Vert_{L^2(B_r \times (0,1))} \leq N_1 \Vert w \Vert_{L^2(B_{R} \times (0,1))}
\end{align}
which proves \eqref{EqInductionHypothesisLemma2.2[EMZ17]} for $\vert\gamma\vert=0$.

\textit{Step 2: Induction step.} For the induction step, we let $\gamma\in\N^n$ with $\vert\gamma\vert = l+1$ and we assume that the estimate \eqref{EqInductionHypothesisLemma2.2[EMZ17]} holds for all multi-indices $\xi\in\N^n$ with $\vert\xi\vert \leq l$. We differentiate the equation to infer the equation for $\partial_x^\gamma w$, which is given by
\begin{align*}
\begin{cases}
\begin{alignedat}{2}
(\p_t - \Delta) \p_x^{\gamma} w &= 0 \quad &&\mbox{in } \Omega \times (0,\infty),\\
\partial_x^\gamma w & = 0 \quad &&\mbox{on } \Omega \times \{0\}.
\end{alignedat}
\end{cases}
\end{align*}
Invoking again Lemma A.5 from \cite{EMZ17} we deduce that
\begin{align}\label{Eq2Lemma2.2[EMZ17]}
\Vert \partial_t \partial_x^\gamma w \Vert_{L^2(B_r \times (0,1))} + \Vert D_x^2 \partial_x^\gamma w \Vert_{L^2(B_r \times (0,1))} &\leq N_0 \delta^{-2} \Vert \partial_x^\gamma w \Vert_{L^2(B_{r+\delta} \times (0,1))}
\end{align}
We distinguish two cases to estimate the right hand side.

\textit{Case 1:} When $1 \leq \vert \gamma \vert \leq 2$, choose $\delta = (R-r)/2$ in \eqref{Eq2Lemma2.2[EMZ17]} to deduce from \eqref{Eq1Lemma2.2[EMZ17]} (with $r$ substituted by $r+\delta$) that
\begin{equation}\label{Eq3Lemma2.2[EMZ17]}
\begin{aligned}
N_0 \delta^{-2} \Vert \partial_x^\gamma w \Vert_{L^2(B_{r+\delta} \times (0,1))} &\leq \frac{4N_0}{(R-r)^2} \frac{2^{\vert\gamma\vert} N_1}{(R-r)^{\vert\gamma\vert}} \Vert w \Vert_{L^2(B_R \times (0,1))}\\
&\leq \frac{1}{D} \frac{N_1}{(R-r)^2} \Vert w \Vert_{L^2(B_R \times (0,1))} \left(\frac{N_2}{R-r}\right)^{\vert\gamma\vert} \vert\gamma\vert!,
\end{aligned}
\end{equation}
when $N_2 \geq \max\{ 8DN_0, 1 \}$. This estimates the right hand side in the case of $1 \leq \vert \gamma \vert \leq 2$.

\textit{Case 2:} When $\vert\gamma\vert >2$ we observe that there is a multi-index $\xi\in\N^n$ with $2+\vert\xi\vert = \vert\gamma\vert$ and $\Vert \partial_x^\gamma w \Vert_{L^2(B_r \times (0,1))} \leq \Vert D_x^2 \partial_x^\xi w \Vert_{L^2(B_r \times (0,1))}$. By the induction hypothesis \eqref{EqInductionHypothesisLemma2.2[EMZ17]} it follows by choosing $\delta = (R-r)/\vert\gamma\vert$ that
\begin{equation}\label{Eq4Lemma2.2[EMZ17]}
\begin{aligned}
N_0 \delta^{-2} \Vert \partial_x^\gamma w &\Vert_{L^2(B_{r+\delta} \times (0,1))} \leq N_0 \delta^{-2} \Vert D_x^2 \partial_x^\xi w \Vert_{L^2(B_{r+\delta} \times (0,1))}\\
&\leq \frac{N_0 \vert\gamma\vert^2}{(R-r)^2} \frac{N_1}{(R-(r+\delta))^2} \Vert w \Vert_{L^2(B_R \times (0,1))} \left( \frac{N_2}{(R-(r+\delta))} \right)^{\vert\gamma\vert-2} (\vert\gamma\vert - 2)!\\
&\leq \frac{N_0 \vert\gamma\vert^2}{(R-r)^2} \frac{N_1}{(R-r)^2} \Vert w \Vert_{L^2(B_R \times (0,1))} \left( \frac{N_2}{R-r} \right)^{\vert\gamma\vert-2} \left( \frac{\vert\gamma\vert}{\vert\gamma\vert - 1} \right)^{\vert\gamma\vert} (\vert\gamma\vert - 2)!\\
&\leq \frac{1}{D} \frac{N_1}{(R-r)^2} \Vert w \Vert_{L^2(B_R \times (0,1))} \frac{N_2^{\vert\gamma\vert - 1}}{(R-r)^{\vert\gamma\vert}} \vert\gamma\vert^2 (\vert\gamma\vert - 2)!\\
&\leq \frac{1}{D} \frac{N_1}{(R-r)^2} \Vert w \Vert_{L^2(B_R \times (0,1))} \left( \frac{N_2}{R-r} \right)^{\vert\gamma\vert} \vert\gamma\vert!.
\end{aligned}
\end{equation}
For the third inequality we have used that $R-(r+\delta) = \frac{\vert\gamma\vert - 1}{\vert\gamma\vert} (R-r)$. The fourth inequality holds if $N_2 \geq \max\{ DN_0M, 1 \}$ where $\infty > M \geq \sup_{m\in\N} \left( \frac{m}{m-1} \right)^{m}$, and for the last inequality we assumed that $N_2 \geq 2$ since $\vert\gamma\vert^2 (\vert\gamma\vert - 2)! \leq 2 \vert\gamma\vert!$ for $\vert\gamma\vert \geq 2$.

\textit{Conclusion of the induction step:} Bringing together \eqref{Eq2Lemma2.2[EMZ17]} and \eqref{Eq3Lemma2.2[EMZ17]} or \eqref{Eq4Lemma2.2[EMZ17]}, respectively, we deduce
\begin{align}\label{Eq6Lemma2.2[EMZ17]}
\Vert \partial_t \partial_x^\gamma w \Vert_{L^2(B_r \times (0,1))} + \Vert D_x^2 \partial_x^\gamma w \Vert_{L^2(B_r \times (0,1))} \leq \frac{1}{D} \frac{N_1}{(R-r)^2} \Vert w \Vert_{L^2(B_R \times(0,1))} \left( \frac{N_2}{R-r} \right)^{\vert\gamma\vert} \vert\gamma\vert!.
\end{align}
Additionally, from \eqref{Eq3Lemma2.2[EMZ17]} or \eqref{Eq4Lemma2.2[EMZ17]}, respectively, we find (using $\delta = (R-r)/\vert\gamma\vert$)
\begin{align}\label{Eq7Lemma2.2[EMZ17]}
\Vert \partial_x^\gamma w \Vert_{L^2(B_r \times (0,1))} \leq \frac{N_1}{D} \Vert w \Vert_{L^2(B_R \times(0,1))} \left( \frac{N_2}{R-r} \right)^{\vert\gamma\vert} \vert\gamma\vert!.
\end{align}
Then by the Gagliardo-Nirenberg interpolation inequality and using \eqref{Eq6Lemma2.2[EMZ17]} and \eqref{Eq7Lemma2.2[EMZ17]} it follows that for some constant $K>0$
\begin{equation}\label{Eq8Lemma2.2[EMZ17]}
\begin{aligned}
\Vert D_x \partial_x^\gamma w \Vert_{L^2(B_r \times (0,1))} &\leq K \Vert D_x^2 \partial_x^\gamma w \Vert_{L^2(B_r \times (0,1))}^{1/2} \Vert \partial_x^\gamma w \Vert_{L^2(B_r \times(0,1))}^{1/2} + K \Vert \partial_x^\gamma w \Vert_{L^2(B_r \times(0,1))}\\
&\leq \frac{2K}{D} \frac{N_1}{(R-r)} \Vert w \Vert_{L^2(B_R \times (0,1))} \left( \frac{N_2}{R-r} \right)^{\vert\gamma\vert} \vert\gamma\vert! 
\end{aligned}
\end{equation}
By \eqref{Eq6Lemma2.2[EMZ17]}, \eqref{Eq7Lemma2.2[EMZ17]} and \eqref{Eq8Lemma2.2[EMZ17]} we finally deduce
\begin{align*}
(R-r)^2 &\Vert \partial_t \partial_x^\gamma w \Vert_{L^2(B_r \times (0,1))} + \sum_{\alpha\in\N^n, \vert\alpha\vert\leq2} (R-r)^{\vert\alpha\vert} \Vert \partial_x^\alpha \partial_x^\gamma w \Vert_{L^2(B_r \times (0,1))}\\
&\leq (R-r)^2 \Vert \partial_t \partial_x^\gamma w \Vert_{L^2(B_r \times (0,1))} + \Vert \partial_x^\gamma w \Vert_{L^2(B_r \times (0,1))} + (R-r) \Vert D_x \partial_x^\gamma w \Vert_{L^2(B_r \times (0,1))}\\
&\quad + (R-r)^2 \Vert D_x^2 \partial_x^\gamma w \Vert_{L^2(B_r \times (0,1))}\\
&\leq N_1 \Vert w \Vert_{L^2(B_R \times (0,1))} \left( \frac{N_2}{R-r} \right)^{\vert\gamma\vert} \vert\gamma\vert!
\end{align*}
since $D$ was defined as $D = 2K +2$. Thus, \eqref{EqInductionHypothesisLemma2.2[EMZ17]} is proven and the result follows by induction.
\end{proof}

\begin{lemma}\label{LemmaLInftyEstimateHeatEq}
Let $\Omega$, $\Omega'$, $w$ be as in Lemma \ref{Lemma2.2[Emz17]}. Then for all $\gamma\in\N^n$,
\begin{align*}
\vert \partial_x^\gamma w(x_0,t) \vert \leq N \|w\|_{L^2(\Omega \times [0,1])} \rho^{-|\gamma|}|\gamma|! \quad \mbox{for } x_0\in\Omega', t\in[0,1)
\end{align*}
for some $N = N(n, \Omega', \Omega)$ and some $\rho = \rho(n, \Omega', \Omega)\in (0,1)$.
\end{lemma}

\begin{proof}
This result is an easy consequence of Lemma \ref{Lemma2.2[Emz17]} using Sobolev embeddings. First of all we note that by the estimate in Lemma \ref{Lemma2.2[Emz17]} we have $\partial_x^\gamma w \in W^{1,2}([0,1]; L^2(\Omega'))$ for any $\gamma\in\N^n$. Then by Sobolev embedding we have $\partial_x^\gamma w \in C^{0,1/2}([0,1];L^2(\Omega'))$ and
\begin{equation}\label{Eq1LemmaLInftyEstimateHeatEq}
\begin{aligned}
\max_{t\in[0,1]} \Vert \partial_x^\gamma w (\cdot,t) \Vert_{L^2(\Omega')} &\leq C_1 \Vert \partial_x^\gamma w \Vert_{W^{1,2}([0,1];L^2(\Omega'))} = C_1 \left( \Vert \partial_x^\gamma w \Vert_{L^2(\Omega' \times [0,1])} + \Vert \partial_t \partial_x^\gamma w \Vert_{L^2(\Omega' \times [0,1])} \right)\\
&\leq C_1 N \Vert w \Vert_{L^2(\Omega \times (0,1))} \rho^{-\vert\gamma\vert} \vert\gamma\vert!,
\end{aligned}
\end{equation}
where the last inequality follows by the estimate in Lemma \ref{Lemma2.2[Emz17]}.\\
Applying again Sobolev embedding with respect to the space variable with $k>n/2$ fixed and using \eqref{Eq1LemmaLInftyEstimateHeatEq}, we deduce
\begin{align*}
\max_{t\in[0,1]} \Vert \partial_x^\gamma w(\cdot,t) &\Vert_{L^\infty(\Omega')}^2 \leq \max_{t\in[0,1]} C_2 \Vert \partial_x^\gamma w(\cdot,t) \Vert_{W^{k,2}(\Omega')}^2 \leq \max_{t\in[0,1]} C_2 \sum_{\vert\alpha\vert \leq k} \Vert \partial_x^\alpha \partial_x^\gamma w(\cdot,t) \Vert_{L^2(\Omega')}^2\\
&\leq C_2 \sum_{\vert\alpha\vert \leq k} \max_{t\in[0,1]} \Vert \partial_x^{\gamma+\alpha} w(\cdot,t) \Vert_{L^2(\Omega')}^2\\
&\leq C_2 \big\vert \{ \alpha\in\N^n: \vert\alpha\vert \leq k \} \big\vert \ C_1 N \Vert w \Vert_{L^2(\Omega \times (0,1))} \rho^{-(\vert\gamma\vert+k)} (\vert\gamma\vert + k)!\\
&\leq \widetilde{N} \Vert w \Vert_{L^2(\Omega \times (0,1))} \widetilde{\rho}^{\ -\vert\gamma\vert} \vert\gamma\vert!
\end{align*}
for appropriately chosen $\widetilde{N}$ and $\widetilde{\rho}$, and the statement follows.
\end{proof}

\printbibliography

\end{document}